\documentclass[a4paper,reqno,11pt]{amsart}
\usepackage{amsmath, amsfonts, amssymb, amsthm, amscd}
\usepackage{graphicx}
\usepackage{psfrag}
\usepackage{perpage}
\usepackage{url}
\usepackage{color}
\usepackage{xcolor}
\usepackage{mathrsfs}
\usepackage{stmaryrd}
\usepackage{mathabx}
\usepackage{scalerel} 
\usepackage{tikz,tikz-cd}\usepackage{caption,subcaption}
\usetikzlibrary{arrows,decorations.pathmorphing,backgrounds,positioning,fit,petri}

\usetikzlibrary{decorations.pathreplacing}
\usepackage{subdepth}

\usepackage{dsfont} 

\usepackage[utf8]{inputenc}
\usepackage[T1]{fontenc}
\usepackage{microtype}
\usepackage{mathtools}
\usepackage[a4paper,scale={0.77,0.74},marginratio={1:1},footskip=7mm,headsep=10mm]{geometry}

\usepackage[linktocpage]{hyperref}
\definecolor{dukeblue}{rgb}{0.0, 0.0, 0.61}

\hypersetup{
	colorlinks=true, linkcolor=dukeblue,
	citecolor=dukeblue
}
\usepackage{fancyhdr}
\usepackage{nameref}

\makeatletter
\def\@secnumfont{\bfseries}

\def\section{\@startsection{section}{1}%
  \z@{.7\linespacing\@plus\linespacing}{.5\linespacing}%
  {\normalfont\large\bfseries\centering}}

\def\subsection{\@startsection{subsection}{2}%
  \z@{.5\linespacing\@plus.7\linespacing}{-.5em}%
  {\normalfont\bfseries}}

\def\subsubsection{\@startsection{subsubsection}{3}%
  \z@{.5\linespacing\@plus.7\linespacing}{-.5em}%
  {\normalfont}}

\def\specialsection{\@startsection{section}{1}%
  \z@{\linespacing\@plus\linespacing}{.5\linespacing}%
  {\normalfont\centering\large\bfseries}}
\makeatother

\makeatletter

\renewenvironment{proof}[1][\proofname]{\par
\pushQED{\qed}%
\normalfont \topsep4\p@\@plus4\p@\relax
\trivlist
\item[\hskip\labelsep
\bfseries
#1\@addpunct{.}]\ignorespaces
}{%
\popQED\endtrivlist\@endpefalse
}
\makeatother

\makeatletter
\newcommand \Dotfill {\leavevmode \leaders \hb@xt@ 6pt{\hss .\hss }\hfill \kern \z@}
\makeatother

\makeatletter
\def\@tocline#1#2#3#4#5#6#7{\relax
  \ifnum #1>\c@tocdepth 
  \else
    \par \addpenalty\@secpenalty\addvspace{#2}%
    \begingroup \hyphenpenalty\@M
    \@ifempty{#4}{%
      \@tempdima\csname r@tocindent\number#1\endcsname\relax
    }{%
      \@tempdima#4\relax
    }%
    \parindent\z@ \leftskip#3\relax \advance\leftskip\@tempdima\relax
    \rightskip\@pnumwidth plus4em \parfillskip-\@pnumwidth
    #5\leavevmode\hskip-\@tempdima
      \ifcase #1
       \or\or \hskip 1.65em \or \hskip 3.3em \else \hskip 4.95em \fi%
      #6\nobreak\relax
    \Dotfill
    \hbox to\@pnumwidth{\@tocpagenum{#7}}\par
    \nobreak
    \endgroup
  \fi}
\makeatother

\makeatletter
\def\l@section{\@tocline{1}{0pt}{1pc}{}{}}
\renewcommand{\tocsection}[3]{%
\indentlabel{\@ifnotempty{#2}{\ignorespaces#1 #2.\hskip 0.7em}}#3}
\def\l@subsection{\@tocline{2}{0pt}{1pc}{5pc}{}}

\def\l@subsubsection{\@tocline{3}{0pt}{1pc}{7pc}{}}

\makeatother

\numberwithin{equation}{section}

\newtheoremstyle{mytheorem}{.7\linespacing\@plus.3\linespacing}{.7\linespacing\@plus.3\linespacing}%
     {\itshape}
     {}
     {\bfseries}
     {. }
     {0.3ex}
     {\thmname{{\bfseries #1}}\thmnumber{ {\bfseries #2}}\thmnote{ (#3)}}  

\theoremstyle{mytheorem}

\newtheorem{theorem}{Theorem}[section]
\newtheorem{lemma}[theorem]{Lemma}
\newtheorem{proposition}[theorem]{Proposition}

\newcommand{\bbE}{{\ensuremath{\mathbb E}} }

\newcommand{\bbN}{{\ensuremath{\mathbb N}} }

\newcommand{\bbP}{{\ensuremath{\mathbb P}} }

\newcommand{\bbR}{{\ensuremath{\mathbb R}} }

\newcommand{\bbZ}{{\ensuremath{\mathbb Z}} }

\newcommand{\cC}{{\ensuremath{\mathcal C}} }

\newcommand{\cS}{{\ensuremath{\mathcal S}} }

\newcommand{\cU}{{\ensuremath{\mathcal U}} }

\newcommand{\cZ}{{\ensuremath{\mathcal Z}} }

\newcommand{\gb}{\beta}

\newcommand{\gd}{\delta}

\newcommand{\gt}{\vartheta}

\newcommand{\gl}{\lambda}

\newcommand{\gs}{\sigma}

\newcommand{\go}{\omega}

\DeclareMathSymbol{\leqslant}{\mathalpha}{AMSa}{"36} 
\DeclareMathSymbol{\geqslant}{\mathalpha}{AMSa}{"3E} 
\DeclareMathSymbol{\eset}{\mathalpha}{AMSb}{"3F}     

\newcommand{\sumtwo}[2]{\sum_{\substack{#1 \\ #2}}} 

\newcommand{\R}{\mathbb{R}}

\newcommand{\Z}{\mathbb{Z}}
\newcommand{\N}{\mathbb{N}}

\def\bs{\boldsymbol}

\newcommand{\PEfont}{\mathrm}

\newcommand{\p}{\ensuremath{\PEfont P}}

\newcommand{\E}{\ensuremath{\PEfont  E}}
\renewcommand{\P}{\p}

\newcommand\bE{\ensuremath{\bs{\mathrm{E}}}}

\DeclareMathOperator{\bbcov}{\ensuremath{\mathbb{C}ov}}

\newcommand{\ind}{\mathds{1}}

\renewcommand{\epsilon}{\varepsilon}
\renewcommand{\theta}{\vartheta}
\renewcommand{\rho}{\varrho}

\newcommand\qa{\bigg[}
\newcommand\qb{\bigg]}
\newcommand\ra{\left(}
\newcommand\rb{\right)}

\newcommand\mI{\mathscr{I}}

\newenvironment{myenumerate}{%
\renewcommand{\theenumi}{\arabic{enumi}}%
\renewcommand{\labelenumi}{{\rm(\theenumi)}}%
\begin{list}{\labelenumi}
	{%
	\setlength{\itemsep}{0.4em}%
	\setlength{\topsep}{0.5em}%
	\setlength\leftmargin{2.45em}%
	\setlength\labelwidth{2.05em}%
	\setlength{\labelsep}{0.4em}%
	\usecounter{enumi}%
	}%
	}%
{\end{list}
}

{\end{list}
}

{\end{list}
}

{\end{myenumerate}}

\newenvironment{myitemize}{%
\begin{list}{$\bullet$}%
 	{%
	\setlength{\itemsep}{0.4em}%
	\setlength{\topsep}{0.5em}%
	\setlength\leftmargin{2.65em}%
	\setlength\labelwidth{2.65em}%
	\setlength{\labelsep}{0.4em}%
	}%
	}%
{\end{list}}

{\end{myitemize}}

\MakePerPage[2]{footnote} 

\date{\today}

\newcommand\dd{\mathrm{d}}

\newcommand\sfp{\mathsf p}

\newcommand\sfr{\mathsf r}

\newcommand\bx{\boldsymbol{x}}
\newcommand\by{\boldsymbol{y}}

\definecolor{cadmiumgreen}{rgb}{0.0, 0.42, 0.24}
\definecolor{red(munsell)}{rgb}{0.95, 0.0, 0.24}

\usepackage{mathtools}

\newcommand{\f}{\frac}
\newcommand{\dlog}{\log\log}

\author{Ziyang Liu}
\address{Department of Mathematics\\
University of Warwick\\
Coventry CV4 7AL, UK}
\email{Ziyang.Liu.1@warwick.ac.uk}

\author{Nikos Zygouras}
\address{Department of Mathematics\\
University of Warwick\\
Coventry CV4 7AL, UK}
\email{N.Zygouras@warwick.ac.uk}

\begin{document}
\title{On the moments of the mass of shrinking balls under the Critical $2d$ Stochastic Heat Flow }
\maketitle

\begin{abstract}
    The Critical $2d$ Stochastic Heat Flow (SHF) is a measure valued stochastic process on $\R^2$ that
    defines a non-trivial solution to the two-dimensional stochastic heat equation with multiplicative 
    space-time noise. Its one-time marginals are a.s. singular with respect to the Lebesgue measure, 
    meaning that the mass they assign to shrinking balls decays to zero faster than their Lebesgue volume.
    In this work we explore the intermittency properties of the Critical 2d SHF by studying the  asymptotics of the  $h$-th moment 
    of the mass that it assigns to shrinking balls of radius $\epsilon$ and we determine that its ratio to the Lebesgue volume
    is of order $(\log\tfrac{1}{\epsilon})^{{h\choose 2}}$ up to possible lower order corrections.
    
\end{abstract}


\section{Introduction}

 The Critical $2d$ Stochastic Heat Flow (SHF) was constructed in \cite{CSZ23a} as a non-trivial, i.e. non-constant and non-gaussian, solution
 to the ill-posed two-dimensional Stochastic Heat Equation (SHE)
\begin{align}\label{SHE}
    \partial_t u = \frac{1}{2}\Delta u+\gb \xi u, \qquad t>0, \, x\in \R^2,
\end{align}
where $\xi$ is a space-time white noise. 
We refer to reviews \cite{CSZ24, CSZ26} for an account of the larger context
and developments in the study of the model.

The solution to \eqref{SHE} lives in the space of generalised functions and, therefore, 
multiplication is a priori not defined. So in order to construct a solution one has to first regularise the equation. One way to do so is
by mollification of the noise $\xi_\epsilon(t,x):=\frac{1}{\epsilon^2}\int_{\R^2} j\big( \frac{x-y}{\epsilon}\big) \xi(t, \dd y)$, so that \eqref{SHE} admits a smooth solution $u^\epsilon$, which in fact can also be represented by a Feynman-Kac formula as
\begin{align}\label{eq:FK}
u^\epsilon(t,x)=\bE_x\Big[ \exp\Big( \beta \int_0^t \xi_\epsilon(t-s, B_s) \dd s - 
\frac{\beta^2 t}{2}\|j_\epsilon\|_{L^2(\R^2)}^2 \Big) \Big],
\end{align}
with $B_s$ being a two-dimensional Brownian motion whose expectation when starting from $x\in \R^2$ is denoted by $\bE_x$ and $j_\epsilon(x):=\frac{1}{\epsilon^2}j(\frac{x}{\epsilon})$. Then one needs to establish whether a sensible limit can be defined when $\epsilon\to 0$. As we will discuss below, for this to be the case a precise choice of $\beta$ depending on $\epsilon$ will be required.

Another approach is by a discretisation scheme; in particular by a distinguished discretisation of the 
Feynman-Kac formula, which is related to the model of
Directed Polymer in Random Environment (DPRE), \cite{C17, Z24}. 
The latter is determined by its partition function:
\begin{align}\label{Z(x,y)}
    Z_{M,N}^{\beta}(x,y) :=\E\Big[  \exp\big(\sum^{N-1}_{n=M+1}(\gb\go(n,S_n)-\lambda(\gb))\Big) \ind_{\{S_{N}=y\}} \, |\, S_M=x\Big],
\end{align}
where $(S_n)_{n\geq 0}$ is a simple, two-dimensional random walk, whose law and expectation 
are denoted, respectively, by $\P$ and $ \E$ and $(\omega_{n,x})_{n\in \N, x\in \Z^2}$ is a family of i.i.d. random variables 
with mean $0$, variance $1$ and finite log-moment generating function $\lambda(\beta):=\log\bbE\big[ e^{\beta\omega} \big]<\infty$, for $\beta\in \R$, which serves as the discrete analogue of a space-time white noise. 
 The DPRE regularisation was the one followed in the construction of the Critical 2d SHF in \cite{CSZ23a}.   
 
In either of these approaches, the singularity that the noise induces in two dimensions demands a particular choice
of the temperature $\beta$, which modulates the strength of the noise. In the DPRE regularisation, the 
Critical 2d SHF emerges through the choice of $\beta=\beta_N$ determined by
\begin{align}\label{choiceb}
\sigma_N^2:= e^{\lambda(2\beta_N)-2\lambda(\beta_N)}-1 = \frac{\pi}{\log N} \Big(1+\frac{\theta+o(1)}{\log N} \Big),
\end{align}
where $o(1)$ denotes asymptotically negligible corrections as $N\to \infty$. 
We note that scaling \eqref{choiceb} constitutes a {\it critical scaling}: a phase transition takes place at $\hat\beta=\pi$ if $\pi$ is replaced
by a parameter $\hat\beta >0$, see \cite{CSZ17}.
In the continuous approximation,
 $\beta:=\beta_\epsilon$ is
chosen as
\begin{align}\label{beta}
\beta^2_\epsilon=\frac{2\pi}{\log\tfrac{1}{\epsilon}}\Big( 1+\frac{\rho+o(1)}{\log\tfrac{1}{\epsilon}} \Big),
\end{align}
where $\rho$ is given as a function of the parameter $\theta$ in \eqref{choiceb}
and depends also  on the mollifier $j$ in a particular way. We refer to 
equation (1.38) in \cite{CSZ19b} for the precise relation.  
We also refer to \cite{BC98} for the origins of this temperature scaling.

The Critical $2d$ SHF was constructed in \cite{CSZ23a}  as the unique limit of the fields
\begin{align}\label{preSHF}
    \cZ_{N; \,s,t}^\beta(\dd x, \dd y):=
    \f{N}{4}
    Z _{[  Ns ],[ Nt ]}^{\gb_N}
    \ra
    \llbracket   \sqrt{N}x    \rrbracket
    ,
    \llbracket   \sqrt{N}y   \rrbracket
    \rb
    \dd x \dd y, \qquad 0\le s<t<\infty \,,
\end{align}
where $[\cdot]$ maps a real number to its nearest, even integer neighbour, $\llbracket\cdot\rrbracket$ maps $\bbR^2$ points to 
their nearest, even integer point on $\bbZ^2_\text{even}:=\{ (z_1,z_2)\in\bbZ^2:z_1+z_2\in 2\bbZ \}$, and $\dd x  \dd y$ is the Lebesgue measure on $\bbR^2 \times \bbR^2$.
More precisely,
\begin{theorem}[\cite{CSZ23a}]
    Let $\gb_N$ be as in (\ref{beta}) for some fixed $\gt\in\bbR$ and $\big( \cZ_{N; \,s,t}^\beta(\dd x, \dd y) \big)_{0\le s<t<\infty}$ 
    be defined as in \eqref{preSHF}. Then, as $N\rightarrow\infty$, the process of random measures 
    $(\cZ_{N; s,t}^{\gb}(\dd x,\dd y))_{0\le s\le t<\infty}$ converges in finite dimensional distributions to a unique limit
    \begin{align*}
        \mathscr{Z}^\gt
        =
        (\mathscr{Z}_{s,t}^{\gt}(\dd x,\dd y))_{0\le s\le t<\infty},
    \end{align*}
    named the Critical 2d Stochastic Heat Flow.
\end{theorem}
$\mathscr{Z}^\gt$ is a {\it measure valued} stochastic process (flow). In fact, its one-time marginals 
\begin{align}\label{one-time}
\mathscr{Z}_t^\theta(\ind, \dd y):=\int_{x\in \R^2} \mathscr{Z}_{0,t}^\theta(\dd x, \dd y)
 \stackrel{d}{=}
 \mathscr{Z}_t^\theta(\dd x, \ind):=\int_{y\in \R^2} \mathscr{Z}_{0,t}^\theta(\dd x, \dd y),
\end{align}
 are singular with respect to Lebesgue: it is proven in \cite{CSZ25} that if
\begin{equation}\label{eq:uniform}
		B(x,\epsilon) := \big\{y \in \R^2 \colon \ |y-x| < \epsilon \big\} \,,
\end{equation}
is the Euclidean ball and 
\begin{align}\label{eq:averaged-SHF}
	\mathscr{Z}_{t}^\theta(B(x,\epsilon))
	:= \int_{y\in B(x,\epsilon) } \mathscr{Z}_t^\theta(\ind, \dd y),
\end{align}
then for any $t > 0$ and $\theta \in \R$, 
\begin{equation}\label{eq:singularity-SHF}
	\text{$\bbP$-a.s.} \qquad
	\lim_{\epsilon \downarrow 0} \; \frac{\mathscr{Z}_{t}^\theta\big(B(x,\epsilon)\big)}{{\rm Vol}(B(x,\epsilon))}
	=\lim_{\epsilon \downarrow 0} \frac{1}{\pi \epsilon^2} \mathscr{Z}_{t}^\theta\big(B(x,\epsilon)\big)
	=0
	\quad \text{for Lebesgue a.e.\ $x\in\R^2$} \,.
\end{equation}
The aim of this work is to investigate the intermittency properties of the Critical 2d SHF  by studying the 
integer moments of the ratio in \eqref{eq:singularity-SHF} and show that, contrary to \eqref{eq:singularity-SHF},
 they grow to infinity as $\epsilon\to0$.
We also determine the growth rate to be a logarithmic power, up to possible sub-logarithmic corrections.
In order to state our result we introduce the notation
\begin{align}\label{Zg}
\mathscr{Z}^\theta_{t}(\varphi):=\int_{\R^2}  \varphi(x) \, \mathscr{Z}_t^\theta(\dd x, \ind),
\end{align}
for any test function $\phi$ on $\R^2$.
Our result then is the following:
\begin{theorem}\label{central}
Let $\cU_{B(0,\epsilon)}(\cdot)$ denote the uniform density on the Euclidean ball of radius $\epsilon$ in $\R^2$:
\begin{equation}\label{eq:uniform}
	\cU_{B(0,\epsilon)}(\cdot) := \frac{1}{\pi \epsilon^2} \, \ind_{B(0,\epsilon)}(\cdot) \qquad
	\text{where } \
	B(0,\epsilon) := \big\{y \in \R^2 \colon \ |y| < \epsilon \big\} \,,
\end{equation}
and let $\mathscr{Z}_{t}^\theta(\cU_{B(0,\epsilon)})$ be defined as in \eqref{Zg} with $\varphi (\cdot)=\cU_{B(0,\epsilon)}(\cdot)$.
For all integer $h\geq 2$, $t>0$ and $\theta\in \R$ there exists a constant $C=C(h,\theta,t)$ such that
\begin{align}\label{thm:mom}
C \big(\log \tfrac{1}{\epsilon}\big)^{h \choose 2} 
\le \bbE\Big[ \Big( \mathscr{Z}_{t}^\theta(\cU_{B(0,\epsilon)}) \Big)^h\Big] 
\le  \big(\log \tfrac{1}{\epsilon}\big)^{{h \choose 2} +o(1) },
\end{align}
with $o(1)$ representing terms that go to $0$ as $\epsilon\to0$.
\end{theorem}
We note that for $h=2$ the correlation structure of the Critical 2d SHF already provides
the sharp asymptotic 
\begin{align}\label{2mom-as}
\bbE\Big[ \Big(  \mathscr{Z}_{t}^\theta(g_{\epsilon^2}) \Big)^2\Big] \sim C_{t} \log\tfrac{1}{\epsilon}, 
\qquad \text{as $\epsilon\to0$},
\end{align}
see relation (1.21) in \cite{CSZ19b}.
\vskip 2mm
Moments of the Critical 2d SHF field can be expressed in terms of the Laplace transform of the total collision time 
of a system of independent Brownian motions with a {\it critical} delta interaction. This is associated to the Hamiltonian
$-\Delta+\sum_{1\leq i<j\leq h} \delta_0(x_i-x_j)$ on $(\R^2)^h$
  known as the {\it delta-Bose gas} \cite{AFHKKL92, DFT94, DR04}; $\delta_0(\cdot)$ 
  is the  Dirac delta-funtion at $0$.
  This operator is singular and ill-defined due to the delta function. To regularize it, one approach, similar 
  to that used for the SHE can be applied, involving a limiting sequence of operators
  $-\Delta+\sum_{1\leq i<j\leq h} \beta_\epsilon^2 \delta_\epsilon(x_i-x_j)$ on $(\R^2)^h$, 
  where $\beta_\epsilon^2$ is as in \eqref{beta} and $\delta_\epsilon$ a mollification of the delta function with a $j_\epsilon$ as in \eqref{eq:FK}. \cite{DFT94} employs, instead, a regularisation in Fourier space.
  The term {\it critical delta interaction} refers to the constant in $\beta_\epsilon^2$
  in \eqref{beta}  being equal to $2\pi$.
 It is well known that independent Brownian motions in dimension $2$ do not meet, however, when their joint measure is
 tilted through a critical delta-attraction between them, then, in the limit when the regularisation is removed, they do meet and have a nontrivial collision time.
 This has been demonstrated in \cite{CM24} (Proposition 5.1), 
 where it has been established that the local collision time in the case of
 two independent Brownian motions (corresponding to $h=2$ in our setting) has a positive {\it log} -- Hausdorff dimension.
 We also refer to works \cite{Ch24a, Ch24b, Ch25b, Ch25c, Ch25d} for the construction of stochastic processes
 from the delta-Bose gas. 
 
Our approach to obtaining the bounds in Theorem \ref{thm:mom} involves expanding the Laplace
transform of the total collision time of $h$ independent Brownian motions 
in terms of diagrams of pairwise interactions (see Figure \ref{fig:CK2}). Estimating a
diagram of this form was first done in \cite{CSZ19b}\footnote{more precisely, in \cite{CSZ19b} the discrete case of independent two-dimensional random walks was treated but the scaling limit recovers the Brownian situation}
 in the case when the starting points of the Brownian motion  are spread out rather than being
concentrated in a $\epsilon$-ball as we study here. Higher-order collision diagrams were estimated in \cite{GQT21}, again
in the situation of spread out initial points, using an alternative approach, which was
 based on resolvent methods and inspired by  \cite{DFT94, DR04}. For {\it sub-critical} delta interactions,
 higher-order collision diagrams of simple two-dimensional random walks were treated in \cite{CZ23, LZ23, LZ24}. 
In particular, in \cite{CZ23}, collision diagrams involving a number of walks growing up to a rate proportional to the square root of the logarithm of the time horizon were analyzed. In all these cases\footnote{\cite{LZ23} addresses a slightly different setting}, collision diagrams express moments of either the stochastic heat 
 equation or the directed polymer model and all of them address scenarios where moments remain
 bounded. In contrast, here we study the situation where moments blow up in the limit as the size of the balls 
 $\epsilon\to0$. 
 \vskip 2mm
 The lower bound in Theorem \ref{central} is reduced to the Gaussian correlation inequality \cite{R14, LM17} --
 a tool already used in the context of the SHE in \cite{F16, CSZ23b}. The upper bound is more demanding as one needs
 to control the complicated recursions emerging from the collision diagrams. Towards this we were guided by the approach of \cite{CZ23}, which was developed to treat the subcritical case. A number of twists have been necessary in order to deal with the 
 singularities of the critical case, which include introducing suitable Laplace multipliers, optimisation  
 and specific combinatorics.
 \vskip 2mm
 Our theorem leaves open whether higher moments grow, in the limit $\epsilon\to0$, proportionally to
$(\log\tfrac{1}{\epsilon})^{h\choose 2}$, i.e. up to a constant factor,
or whether there are sub-logarithmic corrections that lead to  \\
$\tfrac{1}{(\log \tfrac{1}{\epsilon})^{h\choose 2}}\bbE\Big[ \Big( \mathscr{Z}_{t}^\theta(\cU_{B(0,\epsilon)}) \Big)^h\Big] \to \infty$; our upper bound includes corrections of order 
$|\log \epsilon|^{\frac{1}{|\log\log\log \epsilon|}}$.
If the former assertion holds, then, in conjunction with \eqref{2mom-as}, it suggests that pairwise collisions are almost independent even at critical $\delta$–attraction, although they still exhibit a positive correlation. Independence of collision times in the subcritical attraction regime (in a random-walk setting) was established in \cite{LZ24}.
By contrast, the presence of sub-logarithmic corrections would point to a more intricate correlation structure; capturing such behavior would require more refined techniques for deriving lower bounds. In the subcritical case, lower bounds up to negligible errors—within the directed polymer framework and without reliance on the Gaussian Correlation Inequality—were obtained in \cite{CZ24}. These results also reveal a breakdown of the independence phenomenon in the subcritical regime when $h$ grows sufficiently large relative to the polymer scale.
In our setting, it is therefore natural to ask for which threshold $h = h(\varepsilon)$ the asymptotic behavior in \eqref{thm:mom} ceases to hold.

Furthermore, in more recent work \cite{GN25}, a lower bound of the form $e^{e^{h}}$ was established in the critical case when the SHF is averaged over balls of radius $1$. This naturally raises the question of identifying the transition between the asymptotic behavior in \eqref{thm:mom} and the $e^{e^{h}}$ growth observed by Ganguly–Nam for 
$\bbE\Big[ \Big( \mathscr{Z}_{t}^\theta(\cU_{B(0,r)}) \Big)^h\Big]$ as $r$ between 
 $o(1)$ and $O(1)$ scales.
A deeper understanding of these “phase transition” phenomena would be very interesting, and we hope to investigate them in future work.
\vskip 2mm
Before closing this introduction let us make a connection between our results and the notions of
{\it intermittency} and  {\it multifractality}. These notions are closely related but they are not identical; 
see for example \cite{KKX17} and further references therein.

Intermittency (see \cite{CM94})
 refers to the phenomenon of a random object taking very high values with rather small probability. This is
 captured by a nonlinear growth of its moments with respect to the order $h$ of the moment. 
In the case of random fields, this often leads to observing sparse, high peaks. 

Multifractality (see \cite{F13, BP24}), on the other hand, refers to the phenomenon of a random measure
 exhibiting a range of scales in the structure of its high peaks.
 The fractal spectrum of a random measure $\mu$ on $\R^d$ is captured by the exponent $\xi(h)$ in
the moment asymptotics
\begin{align*}
\bbE\big[\mu(B(x,\epsilon))^h \big] \sim \epsilon^{\xi(h)},\qquad \text{as $\epsilon\to0$},
\end{align*}
for $h\in [0,1]$. This notion is useful in determining the Hausdorff dimension of the support of the measure and
indicate phenomena of localisation (see \cite{BP24} for further information). 
The measure $\mu$ is said to exhibit multifractality if the exponent $\xi(h)$ is a nonlinear
function of $h$. The distinctive features between this formulation and the analogous formulation of intermittency
are the {\it small ball limit} $\epsilon\to0$ and the range of $h\in (0,1)$ -- for intermittency one is rather interested in 
the case of $\epsilon$ being (typically) fixed and $h\to \infty$.

The result of Ganguly-Nam \cite{GN25} establishes a strong form of intermittency for the Critical 2d SHF, while
our result that
\begin{align}\label{eq:multifract}
\bbE\big[\mathscr{Z}_{t}^\theta\big(B(x,\epsilon)\big)^h\,\big] 
\sim \epsilon^{2h} \,\big(\log \tfrac{1}{\epsilon}\big)^{\frac{h(h-1)}{2}+o(1)}, \qquad 
\text{as $\epsilon\to0$ for $2\le h\in \N$},
\end{align}
may suggest that the Critical 2d SHF exhibits multifractality at a logarithmic scale (the anticipated log-scale is
consistent with the
picture established in \cite{CSZ25} that the Critical 2d SHF is in $\cC^{0-}$).
 It would be interesting
to formulate the (logarithmic) multifractality features of the Critical 2d SHF. In this regard, one would need to 
develop methods complementary to those of the present article, 
which would allow for asymptotics similar to \eqref{eq:multifract} but for fractional moments $h\in [0,1]$.
We conjecture that asymptotic \eqref{eq:multifract} extends to $h\in [0,1]$.
\vskip 2mm
The structure of the paper is as follows. In Section \ref{sec:aux} we recall the expression of moments of the Critical 2d SHF in 
terms of collision diagrams as well as certain asymptotics that we will use. In Section \ref{sec:upper} we prove the upper
bound in Theorem \ref{central} and in Section \ref{sec:lower} the lower bound.

\section{Auxiliary results on moments of the Critical 2d SHF}\label{sec:aux}
In this section we review the already established formulas of the Critical 2d SHF.
The reader can find the derivation and further details at references \cite{CSZ19b, CSZ23a, GQT21}.

The first moment of the Critical 2d SHF is given by
\begin{equation}
	\bbE[\mathscr{Z}^\theta_{s,t}(\dd x, \dd y)]
	= \tfrac{1}{2} \, g_{\frac{1}{2}(t-s)}(y-x) \, \dd x \, \dd y \,,
\end{equation}
where $g_t(x)=\frac{1}{2\pi t}e^{-\frac{|x|^2}{2t}}$ is the two-dimensional heat kernel. 
The covariance of the Critical 2d SHF has the expression
\begin{equation} \label{eq:formula-cov}
\begin{aligned}
	\bbcov[\mathscr{Z}^\theta_{s,t}(\dd x, \dd y), \mathscr{Z}^\theta_{s,t}(\dd x', \dd y')]
	&= \tfrac{1}{2} \, K_{t-s}^\theta(x,x'; y, y') \, \dd x \, \dd y \, \dd x' \, \dd y' \,,
\end{aligned}
\end{equation}
where
\begin{equation}
\label{eq:m2-lim}
\begin{split}
	K_{t}^{\theta}(x,x'; y,y')
	&\,:=\, \pi \: g_{\frac{t}{4}}\big(\tfrac{y+y'}{2} - \tfrac{x+x'}{2}\big)
	\!\!\! \iint\limits_{0<a<b<t} \!\!\! g_a(x'-x) \,
	G_\theta(b-a) \, g_{t-b}(y'-y) \, \dd a \, \dd b \,.
\end{split}
\end{equation}
In the above formula $G_\theta(t)$ is 
the derivative of the Volterra function \cite{A10, CM24}
The exact expression of $G_\theta(t)$  is
\begin{align}\label{eq:Dick}
G_\theta(t)=\int_0^\infty \frac{e^{(\theta-\gamma)s} st^{s-1}}{\Gamma(s+1)} \dd s,
\qquad  t\in (0, \infty), 
\end{align}
where $\gamma:=-\int_0^\infty \log u e^{-u} \dd u \approx 0.577...$ 
is the Euler constant and $\Gamma(s)$ is the Gamma function.
For $t\in (0,1)$, \eqref{eq:Dick} {may also take the form
\begin{align}\label{eq:Dick2}
G_\theta(t)=\int_0^\infty e^{\theta s} f_s(t) \, \dd s,
\end{align}
where $f_s(t)$ coincides with
 the density of the Dickman subordinator $(Y_s)_{s>0}$ -- a jump process with L\'evy measure 
$x^{-1} \ind_{x\in (0,1)} \dd x$, see \cite{CSZ19a}.

The Laplace transform of \eqref{eq:Dick} has a simple form, which will be useful in our analysis and so we record 
it here:
\begin{proposition}\label{LaplaceG}
Let $G_\theta(t)$ be as in \eqref{eq:Dick} for $t>0$. Then for $\lambda> e^{\theta-\gamma}$ we have that
\begin{align*}
\int_0^\infty e^{-\lambda t} G_\theta(t) \,\dd t =  \frac{1}{\log\lambda-\theta + \gamma}.
\end{align*}
\end{proposition}
\begin{proof}
Replacing formula \eqref{eq:Dick} into the Laplace integral and performing the integrations, we obtain:
\begin{align*}
        \int_0^\infty G_\theta(t) e^{-\lambda t} \,\dd t
        & = \int_0^\infty \int_0^\infty  \frac{e^{(\theta-\gamma) s} t^{s-1}}{\Gamma(s)} e^{-\lambda t} \,\dd s \,\dd t \\
        & = \int_0^\infty  \left( \int_0^\infty   t^{s-1} e^{-\lambda t} \dd t \right)
        \frac{e^{(\theta-\gamma) s} }{\Gamma(s)} \dd s \\
        & = \int_0^\infty  \left(  \frac{1}{\lambda^s}\int_0^\infty  t^{s-1} e^{-t}  \dd t \right)
        \frac{e^{(\theta-\gamma) s} }{\Gamma(s)} \dd s \\
        & = \int^\infty_0\frac{1}{\lambda^s} e^{(\theta-\gamma) s}   \dd s 
        = \int^\infty_0 e^{-(\log\lambda-\theta +\gamma)s} \dd s \\
        &= \frac{1}{\log\lambda-\theta + \gamma}.
\end{align*}
\end{proof}
We will also need the following asymptotics for $G_\theta$, which were established in \cite{CSZ19a}
 \begin{proposition}\label{asymptotic of G}
        For any $\gt\in\bbR$, the function $G_\gt(t)$ is continuous and strictly positive for $t\in(0,1]$. As $t\downarrow 0$ we have the asymptotic,
        $$
        G_\gt(t)=\frac{1}{t(\log\frac{1}{t})^2}\bigg\{ 1+\frac{2\gt}{\log\frac{1}{t}}
        +O\ra \frac{1}{(\log\frac{1}{t})^2} \rb \bigg\}.
        $$
    \end{proposition}

We next move to the formulas for higher moments. These were obtained in \cite{CSZ19b} in the case of the third moment and
in \cite{GQT21} for arbitrary moments.
Here we will adopt the formulation presented in \cite{CSZ19b}.
Let us first write the alluded formula for the $h$-moment
and demystify it afterwards. The formula is:
 \begin{align}\label{mom-formula}
    &\bbE\Big[ \big( \mathscr{Z}_t^\theta(\varphi) \big)^h\Big] 
    = \sum_{m\ge 0}  \,\,\,\,\,\,\, (2\pi)^m \hskip -1.2cm\sumtwo{\{\{i_1,j_1\},...,\{i_m,j_m\} \in \{1,...,h\}^2}
    {\text{with $\{i_k,j_k\} \neq \{i_{k+1},j_{k+1}\}$ for $k=1,...,m-1$}}   \int_{(\R^2)^h} \dd \bx  \, \phi^{\otimes h}(\bx) \notag \\
   &  \iint_{\substack{0\le a_1< b_1<...< a_m < b_m\le t \\ x_1,y_1,...,x_m, y_m \in  \R^2 }}  
     g_{\f{a_1}{2}} (x_1-x^{i_1})  g_{\f{a_1}{2}}  (x_1-x^{j_1})  \,\,
   \prod_{r=1}^m    G_\gt(b_r-a_r) g_{\frac{b_r-a_r}{4}}(y_r-x_r)  \,\, \ind_{\cS_{i_r,j_r}} \notag \\
    &\,\,  \times \Big( \prod_{1\le r\le m-1} g_{\frac{a_{r+1}-b_{\sfp(i_{r+1})}}{2}}(x_{r+1} - y_{\sfp(i_{r+1}) }) \, 
    g_{\frac{a_{r+1}-b_{\sfp(j_{r+1})}}{2}} (x_{r+1}- y_{\sfp(j_{r+1})}) \Big) 
    \dd \vec\bx \,\dd \vec\by \,\dd \vec a \,\dd \vec b \,\,
    \end{align}
    where $\phi^{\otimes h}(\bx):=\phi(x^1)\cdots \phi(x^h)$, 
    $\cS_{i_r,j_r}$ is the event that Brownian motions $i_r$ and $j_r$, only, are involved in collisions in the time interval 
$(a_r,b_r)$ conditioned to both start at positions $x_r$ and ending at positions $y_r$ and for a pair $\{i_r,j_r\}$ we define
\begin{align*}
\sfp(i_r)&:= i_{\ell(r)} \quad\text{with} \quad
\ell(r):=\max\big\{ 0\leq \ell < r \colon \ind_{\cS_{i_\ell,j_\ell}}=1 \,\, \text{and} \,\, i_r\in \{i_\ell, j_\ell\} \big\}
\end{align*}
and similarly for $\sfp(j_r)$.
In other words, $\sfp(i_r)$ is the last time before $r$ that Brownian motion $B^{(i_r)}$ was involved in a collision. We note that if $\sfp(i_r)=0$ then $(b_{\sfp(i_r)},y_{\sfp(i_r)}):=(0,x^{i_r})$.
\vskip 2mm
A diagrammatic representation of formula \eqref{mom-formula} is shown in Figure \ref{fig:CK2}. 
To get a better idea of formula  \eqref{mom-formula} and its diagrammatic representation,   
we may use the Feynman-Kac formula \eqref{eq:FK} from which an easy computation
gives that
\begin{align}\label{momFK}
\bbE\Big[\Big( \int_{\R^2} \phi(x) u^\epsilon(t,x) \dd x \Big)^h \Big]
=\int_{(\R^2)^h} \phi^{\otimes h}(\bx) \, \bE_{\bx}^{\otimes h} \Big[ \Big( \beta_\epsilon^2 \sum_{1\leq i<j\leq h} \int_0^t J_\epsilon(B_s^{(i)} -B_s^{(j)})  \,\dd s\Big) \Big] \,\dd \bx
\end{align}
with $\bx=(x^1,...,x^h)$, $J_\epsilon(x):=\beta_\epsilon^2 \,\frac{1}{\epsilon^2} J\big(\frac{x}{\epsilon}\big)$,
with $J=j*j$ and $j$ as in \eqref{eq:FK}, 
approximates a delta function when $\epsilon\to0$. When $\beta_\epsilon$ is chosen
at the critical value \eqref{beta}, then the main contribution to \eqref{momFK}, in the limit $\epsilon\to0$
comes from configurations where the Brownian motions $B^{(1)},...,B^{(h)}$ have pairwise collisions. 
Expanding the exponential in \eqref{momFK} and breaking down according to when and where the collisions take place 
and which Brownian motions are involved, it gives rise to  formula \eqref{mom-formula} 
and its graphical representation as depicted in Figure \ref{fig:CK2}.  The wiggle lines appearing in that
Figure represent the weights accumulated from collisions of the Brownian motions and we often call it
{\bf replica overlap}.

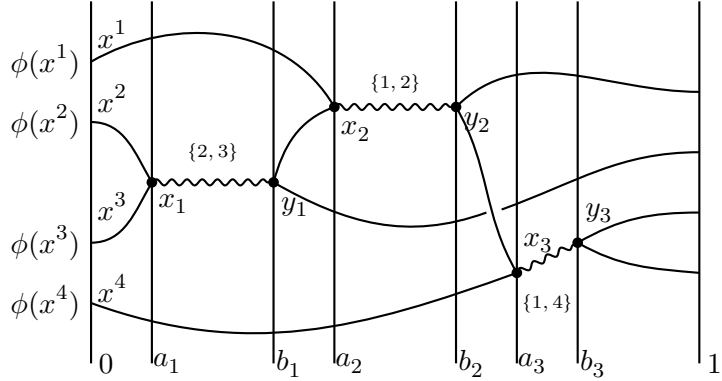
\begin{figure}
    \begin{tikzpicture}[scale=0.4]
        \foreach \i in {6,7,...,8}{
            \draw[-, thick] (2*\i,-6) -- (2*\i,6);
        }
        \draw[-,thick] (2,-6) -- (2,6);
        \draw[-,thick] (0,-6) -- (0,6);
        \draw[-,thick] (6,-6) -- (6,6);
        \draw[-,thick] (8,-6) -- (8,6);
        \draw[-,thick] (20,-6) -- (20,6);

        \fill[black] (2,0) circle [radius=0.175];
        \fill[black] (6,0) circle [radius=0.175];
        \fill[black] (8,2.5) circle [radius=0.175];
        \fill[black] (12,2.5) circle [radius=0.175];
        \fill[black] (14,-3) circle [radius=0.175];
        \fill[black] (16,-2) circle [radius=0.175];

        \draw[thick] (0,4) to [out=30,in=120] (8,2.5);
        \draw[thick] (0,2) to [out=0,in=120] (2,0);
        \draw[thick] (0,-2) to [out=0,in=-120] (2,0);
        \draw[thick] (0,-4) to [out=-20, in=200] (14,-3);
        \draw[thick] (6,0) to [out=75,in=200] (8,2.5); 
        \draw[thick] (6,0) to [out=-30,in=200] (13,-1);
        \draw[thick] (13.5,-0.9) to [out=20,in=180] (20,1);
        \draw[thick] (12,2.5) to [out=-60, in=120] (14,-3);
        \draw[thick] (12,2.5) to [out=45,in=180] (20,3);
        \draw[thick] (16,-2) to [out=30,in=180] (20,-1);
        \draw[thick] (16,-2) to [out=-30,in=175] (20,-3);

        \draw[-,thick,decorate,decoration={snake, amplitude=.4mm,segment length=2mm}] (2,0) -- (6,0);
        \draw[-,thick,decorate,decoration={snake, amplitude=.4mm,segment length=2mm}] (8,2.5) -- (12,2.5);
        \draw[-,thick,decorate,decoration={snake, amplitude=.4mm,segment length=2mm}] (14,-3) -- (16,-2);

        \node at (0.5,-6) {{$0$}};
        \node at (2.5,-6) {{$a_1$}};
        \node at (6.5,-6) {{$b_1$}};
        \node at (8.5,-6) {{$a_2$}};
        \node at (12.5,-6) {{$b_2$}};
        \node at (14.5,-6) {{$a_3$}};
        \node at (16.5,-6) {{$b_3$}};
        \node at (20.5,-6) {{$1$}};

        \node at (-1.5,4) {{$\phi(x^1)$}}; 
        \node at (-1.5,2) {{$\phi(x^2)$}};
        \node at (-1.5,-2) {{$\phi(x^3)$}};
        \node at (-1.5,-4) {{$\phi(x^4)$}};
        
        \node at (0.7,4.9) {{$x^1$}}; 
        \node at (0.7,2.8) {{$x^2$}};
        \node at (0.7,-0.8) {{$x^3$}};
        \node at (0.7,-3.5) {{$x^4$}};

        \node at (2.7,-0.7) {{$x_1$}};     
        \node at (6.7,-0.8) {{$y_1$}}; 
        \node at (8.7,1.7) {{$x_2$}};     
        \node at (12.7,2) {{$y_2$}};
        \node at (14.7,-2) {{$x_3$}};     
        \node at (16.7,-1) {{$y_3$}};

        \node at (4,1) {\tiny {$\{2,3\}$}};
        \node at (10,3.3) {\tiny {$\{1,2\}$}};
        \node at (15,-4) {\tiny {$\{1,4\}$}};
    \end{tikzpicture}
    \caption{This picture supplies a diagrammatic representation of the moment formula \eqref{mom-formula}, 
    more precisely of the term corresponding to $m=3$. The wiggle lines between points $(a_r,x_r)$ and 
    $(b_r,y_r)$ are given weight $ G_\gt(b_r-a_r) g_{\frac{b_r-a_r}{4}}(y_r-x_r)  $, representing the 
    total collision time of Brownian motions $B^{(i_r)} , B^{(j_r)}$ with a critically scaled attractive potential. 
    Pairs $\{i_r,j_r\}$ above wiggle lines indicate the indices of the pair of Brownian motions involved in the
    collisions.
    Solid lines between points $(a_r, x_r)$ and $(a_{\sfp(i_r)}, y_{\sfp(i_r)})$  are weighted by the heat kernel 
   $g_{\frac{a_{r}-b_{\sfp(j_{r})}}{2}} (x_{r}- y_{\sfp(j_{r})})$. 
    }
    \label{fig:CK2}
\end{figure}

\begin{figure}
    \begin{tikzpicture}[scale=0.4]
        \foreach \i in {6,7,...,8}{
            \draw[-, thick] (2*\i,-6) -- (2*\i,6);
        }
        \draw[-,thick] (-2,-6) -- (-2,6);
        \draw[dashed,thick] (0,-6) -- (0,6);
        \draw[-,thick] (2,-6) -- (2,6);
        \draw[-,thick] (6,-6) -- (6,6);
        \draw[-,thick] (8,-6) -- (8,6);
        \draw[-,thick] (20,-6) -- (20,6);

        \fill[black] (-2,0) circle [radius=0.175];
        \fill[black] (2,0) circle [radius=0.175];
        \fill[black] (6,0) circle [radius=0.175];
        \fill[black] (8,2.5) circle [radius=0.175];
        \fill[black] (12,2.5) circle [radius=0.175];
        \fill[black] (14,-3) circle [radius=0.175];
        \fill[black] (16,-2) circle [radius=0.175];

        \draw[thick] (-2,0) to [out=70,in=-150] (0,4);
            \draw[thick] (0,4) to [out=30,in=120] (8,2.5);
        \draw[thick] (-2,0) to [out=60,in=180] (0,2);
            \draw[thick] (0,2) to [out=0,in=120] (2,0);
        \draw[thick] (-2,0) to [out=-60,in=180] (0,-2);
            \draw[thick] (0,-2) to [out=0,in=-120] (2,0);
        \draw[thick] (-2,0) to [out=-70, in=160] (0,-4);
            \draw[thick] (0,-4) to [out=-20, in=200] (14,-3);
        \draw[thick] (6,0) to [out=75,in=200] (8,2.5); 
        \draw[thick] (6,0) to [out=-30,in=200] (13,-1);
        \draw[thick] (13.5,-0.9) to [out=20,in=180] (20,1);
        \draw[thick] (12,2.5) to [out=-60, in=120] (14,-3);
        \draw[thick] (12,2.5) to [out=45,in=180] (20,3);
        \draw[thick] (16,-2) to [out=30,in=180] (20,-1);
        \draw[thick] (16,-2) to [out=-30,in=175] (20,-3);

        \draw[-,thick,decorate,decoration={snake, amplitude=.4mm,segment length=2mm}] (2,0) -- (6,0);
        \draw[-,thick,decorate,decoration={snake, amplitude=.4mm,segment length=2mm}] (8,2.5) -- (12,2.5);
        \draw[-,thick,decorate,decoration={snake, amplitude=.4mm,segment length=2mm}] (14,-3) -- (16,-2);

        

        \node at (-1.5,-6) {{$0$}};
        \node at (0.5,-6) {{$\epsilon^2$}};
        \node at (2.5,-6) {{$a_1$}};
        \node at (6.5,-6) {{$b_1$}};
        \node at (8.5,-6) {{$a_2$}};
        \node at (12.5,-6) {{$b_2$}};
        \node at (14.5,-6) {{$a_3$}};
        \node at (16.5,-6) {{$b_3$}};
        \node at (20.5,-6) {{$1+\epsilon^2$}};

        \node at (0.7,4.9) {{$x^1$}}; 
        \node at (0.7,2.8) {{$x^2$}};
        \node at (0.7,-0.8) {{$x^3$}};
        \node at (0.7,-3.5) {{$x^4$}};

        \node at (2.7,-0.7) {{$x_1$}};     
        \node at (6.7,-0.8) {{$y_1$}}; 
        \node at (8.7,1.7) {{$x_2$}};     
        \node at (12.7,2) {{$y_2$}};
        \node at (14.7,-2) {{$x_3$}};     
        \node at (16.7,-1) {{$y_3$}};

        \node at (4,1) {\tiny {$\{2,3\}$}};
        \node at (10,3.3) {\tiny {$\{1,2\}$}};
        \node at (15,-4) {\tiny {$\{1,4\}$}};

        \node at (-3,0) {{$0^{\otimes h}$}};
    \end{tikzpicture}
    \caption{
    This figure shows a diagrammatic representation of formula \eqref{mom-formula-g}. The laces 
    and wiggle lines are assigned weights similarly to the assignments in Figure \ref{fig:CK2}.
    }\label{Fig2}
\end{figure}
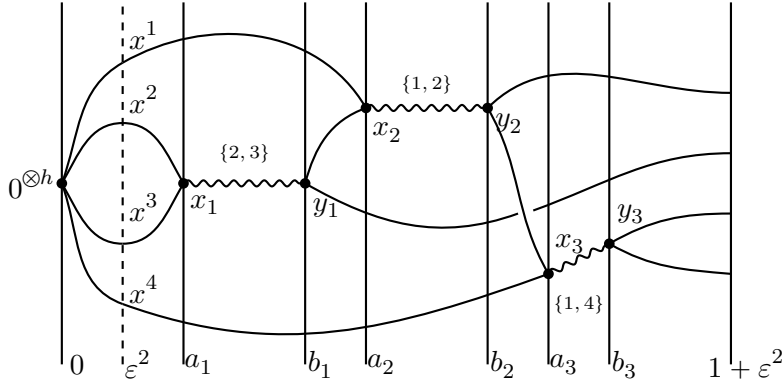
\vskip 2mm
 Our main objective, which will be carried in the next sections, is to determine the asymptotics of 
 \eqref{mom-formula} when the test function $\phi$ is $\cU_{B(0,\epsilon)}(\cdot) := \frac{1}{\pi \epsilon^2} \, \ind_{B(0,\epsilon)}(\cdot)$ .
However, it will be more convenient to work with $\phi$ being a heat kernel approximation of the delta function 
and look into the asymptotics of
 \begin{align}\label{def:Zscr}
 \mathfrak{M}_\epsilon^{\theta,h}
 := \bbE\Big[ \big( \mathscr{Z}_{1}^\theta(g_{\f{\epsilon^2}{2}}) \big)^h\Big] 
 \qquad \text{with}
 \qquad  g_{\f{\epsilon^2}{2}}(x)=\frac{1}{\pi\epsilon^2}e^{-\frac{|x|^2}{\epsilon^2}},
 \end{align}
 and then perform a comparison to $\bbE\Big[ \big( \mathscr{Z}_{1}^\theta( \cU_{B(0,\epsilon)}) \big)^h\Big]$.
For simplicity we just consider time $t=1$.

Let us write the series expression for  $\mathfrak{M}_\epsilon^{\theta,h}$. 
For every $i$, we integrate,
  $g_{\f{\epsilon^2}{2}}(x^i)$ against the heat kernel corresponding to the weight of the lace
  emanating from $x^i$ (see Figure \ref{fig:CK2}):
  \begin{align*}
  \int_{\R^2} g_{\f{\epsilon^2}{2}}(x^i) \, g_{\f{a_{\sfr(i)}}{2}}(x_{\sfr(i)}-x^i) \, \dd x^i 
  =g_{\f{a_{\sfr(i)}+\epsilon^2}{2}}(x_{\sfr(i)}),
  \end{align*}
  where we have denoted by $\sfr(i)$ the index which determines the point $(a_r,x_r), r=1,...,m$ 
  that is connected to $(0,x^i)$. Performing all such integrations over the initial points $x^i, i=1,...,h$
  and shifting the time variables $a_1,b_1,...,a_r,b_r$ by $\epsilon^2$, we arrive at the following formula,
  which is depicted in Figure \ref{Fig2}:
  \begin{align}\label{mom-formula-g}
    & \mathfrak{M}_\epsilon^{\theta,h} 
    = \sum_{m\ge 0}  \,\,\,\,\,\,\, (2\pi)^m \hskip -1.2cm\sumtwo{\{\{i_1,j_1\},...,\{i_m,j_m\} \in \{1,...,h\}^2}
    {\text{with $\{i_k,j_k\} \neq \{i_{k+1},j_{k+1}\}$ for $k=1,...,m-1$}}    \notag \\
   &  \iint_{\substack{\epsilon^2
   \le a_1< b_1<...< a_m < b_m\le 1+\epsilon^2 \\ x_1,y_1,...,x_m, y_m \in \R^2 }}  
    g_{\f{a_1}{2}} (x_1)^2  \,\,
   \prod_{r=1}^m    G_\gt(b_r-a_r) g_{\frac{b_r-a_r}{4}}(y_r-x_r)  \,\, \ind_{\cS_{i_r,j_r}}  \\
    &  \times \Big( \prod_{1\le r\le m-1} g_{\frac{a_{r+1}-b_{\sfp(i_{r+1})}}{2}}(x_{r+1} - y_{\sfp(i_{r+1}) }) \, 
    g_{\frac{a_{r+1}-b_{\sfp(j_{r+1})}}{2}} (x_{r+1}- y_{\sfp(j_{r+1})}) \Big)
     \,\, \dd \vec\bx \,\dd \vec\by \,\dd \vec a \,\dd \vec b \,\,\, ,\notag
    \end{align}
    We note that if $\sfp(i_r)=0$, then $(b_{\sfp(i_r)}, y_{\sfp(i_r)}) =(0,0)$.
\section{Upper bound}\label{sec:upper}
In this section we prove the upper bound in Theorem \ref{central}. The main estimate is contained in
the following proposition:
\begin{proposition}\label{thm: g upper bound}
Recall the definition of $\mathfrak{M}_\epsilon^{\theta,h}$ from \eqref{def:Zscr}.
For any $\gd>0$, $h\ge 2$ and $\gt\in\bbR$, then
    \begin{align}
        \mathfrak{M}_\epsilon^{\theta,h} \le \ra\log\f{1}{\epsilon}\rb^{ {h\choose 2} + o(1) }, \qquad \text{as $\epsilon\rightarrow 0$}.
    \end{align}
\end{proposition}
Having the above estimate at hand we can deduce
 the upper bound in \eqref{thm:mom} as follows:
\begin{proof}[Proof of the upper bound in Theorem \ref{central}]
 We have the comparison: 
 \begin{align*}
     \cU_{B(0,\epsilon)}(\cdot) = \f{1}{\pi\epsilon^2} \ind_{B(0,\epsilon)}(\cdot)
    \le e g_{\epsilon^2/2}( \cdot ).
 \end{align*}
 Hence, by Proposition \ref{thm: g upper bound}, 
\begin{align*}
    \bbE \qa \ra \mathscr{Z}^\gt_1 \ra \cU_{B(0,\epsilon)} \rb \rb^h \qb
    \le e^h \mathfrak{M}_\epsilon^{\theta,h} 
    \le e^h \ra\log\f{1}{\epsilon}\rb^{ {h\choose 2}+o(1) }.
\end{align*} 
as $\epsilon\rightarrow 0$.
 \end{proof}
The rest of the section is devoted to the proof of Proposition  \ref{thm: g upper bound}.
As a warm up computation, we start with the following preliminary estimate on 
 $\mathfrak{M}_\epsilon^{\theta,h}$:
\begin{lemma}\label{lemma: sum over spatial variables}
    For $0<\epsilon<1$, the following estimate holds:
    \begin{align}\label{before replica}
    \mathfrak{M}_\epsilon^{\theta,h}
    & \le \sum_{m\ge 0} \mI_{m,h,\epsilon}
    \end{align}
    where $\mI_{0,h,\epsilon}=1$, $\mI_{1,h,\epsilon}=C{h\choose 2}\log\f{1}{\epsilon}$ for some $C>0$, and for $m\ge 2$:
    \begin{align}\label{eq: I1}
    \mI_{m,h,\epsilon}:=
    {h \choose 2} \left[ {h\choose 2} -1 \right]^{m-1}
    \hskip -.5cm
    \idotsint\limits_{ \sum_i(u_i+v_i)\le 1+\epsilon^2\, ,\, u_1 >\epsilon^2}   
    \hskip -.1cm  
     \frac{1}{u_1}   \prod _{1\le r\le m-1} \frac{G_\gt(v_r)}{\frac{1}{2}(v_r+u_r)+ u_{r+1}} 
     \, G_\gt(v_m) \,\dd \vec u \, \dd \vec v
    \end{align}
\end{lemma}
\color{black}
\begin{proof}
We work with \eqref{mom-formula-g}. 
The $m=0$ term in that formula is simply $1$. The $m=1$ term is equal to:
\begin{align}\label{m1term}
    2\pi \sum_{i,j\in\{1,...,h\}} \iint_{\stackrel{\epsilon^2\le a < b\le 1+\epsilon^2}{x,y\in\bbR^2}}  \quad g_{\f{a}{2}}(x)^2 G_\gt(b-a) g_{\f{b-a}{4}}(y-x) \, \dd x \, \dd y \, \dd a \, \dd b.
\end{align}
To simplify notations, we extend the integral $\iint_{\epsilon^2\le a<b\le 1+\epsilon^2}(...) \dd a \dd b$ to $\iint_{\epsilon^2\le a<b\le 2}(...) \dd a \dd b$. We first perform the integration over $y$, which gives
\begin{align*}
    \int_{\bbR^2} g_{\f{b-a}{4}}(y-x) \dd y=1.
\end{align*}
Then by Proposition \ref{asymptotic of G}, we integrate over $b$:
\begin{align*}
    \int_{a}^2 G_\gt(b-a) \dd b \le C.
\end{align*}
Therefore we bound \eqref{m1term} by:
\begin{align*}
    C \sum_{i,j\in\{1,...,h\}} \iint_{\epsilon^2\le a\le 2 \, , x\in\bbR^2} \quad g_{\f{a}{2}}(x)^2 \, \dd x \dd a
    & =  C {h\choose 2} \iint_{\epsilon^2\le a\le 2 \,, x\in\bbR^2} g_{\f{a}{2}}(x)^2 \, \dd x \dd a \\
    & = C {h\choose 2} \int_{\epsilon^2}^{ 2} g_a(0) \dd a\\
    & = C {h\choose 2} \log (\f{2}{\epsilon^2}) \le C {h\choose 2}\log(\f{1}{\epsilon}).
\end{align*}
Now we treat the case $m\ge 2$. We will follow the convention that $b_0=0$. and recall the convention that if $\sfp(i_r)=0$, then
$(b_{\sfp(i_r)}, y_{\sfp(i_r)}) =(0,0)$. 
We start by performing the integration over $y_m$, which amounts to
\begin{align*}
\int_{\R^2} g_{\frac{b_m-a_m}{4}}(y_m-x_m) \,\dd y_m=1.
\end{align*}
Next we integrate $x_m$. 
\begin{align*}
    & \int_{\bbR^2}
    g_{\frac{a_{m}-b_{\sfp(i_{m})}}{2}} \big(x_{m}- y_{\sfp(i_{m})} \big) \,
    g_{\frac{a_{m}-b_{\sfp(j_{m})}}{2}} \big(x_{m}- y_{\sfp(j_{m})} \big) \, \dd x_m\\
    & = g_{a_m-\frac{1}{2}(b_{\sfp(i_{m})} +b_{\sfp(j_{m})})} 
    \big(y_{\sfp(i_{m})} -y_{\sfp(j_m)} \big) \le 
    \frac{1/\pi}{2a_m-( b_{\sfp(i_{m})}\!+b_{\sfp(j_{m})} )}
    \le \f{1/\pi}{(a_m-b_{m-1})+(a_m-b_{m-2})},
\end{align*}
as $b_{\sfp(i_{m})}$ and $b_{\sfp(j_{m})}$ may be before $b_{m-1}$ and $b_{m-2}$, respectively, but not
after and they cannot be both equal to  just one of $b_{m-1}$ or $b_{m-2}$.

The result then follows by iterating the same integration successively over 
$y_{m-1}, x_{m-1},...,y_1, x_1$ and changing variables as
\begin{align*}
v_i:= b_i-a_i \qquad \text{and} \qquad u_i:= a_i-b_{i-1}.
\end{align*}
The combinatorial factor $ {h \choose 2} \left[ {h\choose 2} -1 \right]^{m-1}$ counts the choices of 
assigning pairs $\{i,j\}$ to the wiggle lines, noting that two consecutive wiggle lines will need to 
have different pairs assigned to them.
\end{proof}
We will next bound \eqref{eq: I1}. The first step is to introduce multipliers and integrate
over the $v_1,...,v_r$ variables to obtain the following intermediate estimate:

\begin{lemma}[Integration of the replica variables]\label{technical1}
There exists a constant $C>0$ such that for all $\lambda>e^{\theta-\gamma}$, it holds
\begin{equation}\label{eq: integration of the replica variables}
    \mathfrak{M}_\epsilon^{\theta,h} \le 
    C e^{2\gl} \sum_{m=0}^\infty  \mI^{(\lambda)}_{m,h,\epsilon}
\end{equation}
where $\mI^{(\lambda)}_{m,h,\epsilon}:=\mI_{m,h,\epsilon}$ for $m=0,1$, and for $m\ge 2$:
\begin{equation}\label{eq: I2}
    \mI^{(\lambda)}_{m,h,\epsilon} := {h \choose 2} \left[ {h\choose 2} -1 \right]^{m-1}
    \idotsint\limits_{\sum_i u_i \le 2\,,\, u_1>\epsilon^2}
    \frac{1}{u_1}
    \prod^m_{r=2}F_\gl \big(u_r+ \frac{u_{r-1}}{2} \big) \,\, \dd\vec u , 
\end{equation}
with
\begin{equation}\label{G}
F_\gl(w):=
\int^\infty_0\frac{ e^{- \gs w} \, \dd \sigma}{\log(\gl+\gs/2)-\theta+\gamma}.
\end{equation}

\end{lemma}
\begin{proof}
For $m<2$ there is nothing to prove. For $m\ge 2$, to simplify notationally, we extend the integral in \eqref{before replica} to $\sum_{i}(u_i+v_i)<2$.
We next introduce the multipliers. 
To this end, we consider a parameter $\gl>0$, which will be suitably chosen later on
 and we multiply \eqref{before replica} by $e^{2\gl} e^{-\gl\sum_i v_i} \geq 1$
to obtain the bound
\begin{equation}
\begin{split}
    \mI_{m,h,\epsilon}
    & \le e^{2\lambda} 
    {h \choose 2} \left[ {h\choose 2} -1 \right]^{m-1} \hskip -.5cm
    \idotsint\limits_{\sum_i(u_i+v_i)\le 2\,,\, u_1>\epsilon^2}  
    \hskip -.5cm  \frac{1}{u_1}
    \prod _{1\le r\le m-1} \frac{e^{-\gl v_r} \, G_\gt(v_r)}{\frac{1}{2}(v_r+u_r)+u_{r+1}}  \,
     G_\gt(v_m) \,\, \dd\vec u \dd\vec v .
\end{split}
\end{equation}
Next we integrate all the $v$ variables. Starting from $v_m$ we  use the bound
\begin{equation}\label{a}
        \int_0^2 G_\gt(v_m)e^{-\gl v_m} \dd v_m
        \le    \int_0^2 G_\gt(v_m) \dd v_m \le C
    \end{equation}
which follows from Proposition \ref{asymptotic of G}. For the rest of the $v$-variables we use 
that for any $w>0$ we have the bound:
    \begin{align}\label{GG-Lapl}
        \int_0^2 
        \frac{e^{-\gl v} \, G_\gt(v)}{v/2+w}  \,\dd v     
        & =  \int^2_0   \int^\infty_0e^{-\gs(v/2+w)} e^{-\gl v}G_{\gt}(v) \, \dd \gs \, \dd v \notag\\
        & \le 
        \int^\infty_0 \dd\gs \, e^{- \gs w}
        \int^\infty_0 e^{-(\gl+\gs/2) v} \,G_{\gt}(v)  \dd v  \notag \\
        & = \int^\infty_0  \frac{ e^{- \gs w} \, \dd\gs}{\log(\gl+\gs/2)-\gt+\gamma}, 
    \end{align}
where in the last step we used Proposition \ref{LaplaceG}. For the last formula to be valid we need,
according to Proposition \ref{LaplaceG}, to choose $\lambda> e^{\theta-\gamma}$.
To conclude,  we choose $w:=u_{r}+\frac{1}{2}u_{r-1}$ and insert successively for $r=2,...,m$.
\end{proof}

The next step is to integrate over the $u$ variables in \eqref{eq: I2}. Our approach here is
inspired by  \cite{CZ23}. 
However, some details are rather different as we make use of the multiplier $\lambda$ and
we also take into account the critical nature of the Critical 2d SHF. 

To start with we define:
\begin{equation}\label{eq: def of f}
    f_\gl(w):= \int^{2}_w F_\gl(v) \dd v = \int^\infty_0 
     \frac{1}{\gs} \frac{e^{-\gs w}-e^{-2\gs}}{\log(\gl+\gs/2)-(\theta-\gamma)}  \, \dd\gs .
    \end{equation}
Note that $f'_\lambda=-F_\lambda\leq 0$ on $(0,2]$, as $F$ is non-negative, thus, $f$ is non-increasing.
We have the following Lemma:
\begin{lemma}\label{lemma:fF}
    There exists $C>0$ such that for all $\lambda> \ra e^{2(\theta-\gamma)} \vee 1 \rb$
    and $w\in (0,1)$ we have:
    \begin{align}\label{fF:eq1}
        \int_0^{2}  F_\gl(u+w)  f_\gl(u)^j \,\dd u
     \le \sum^{j+1}_{\ell=0}
       \frac{j!}{(j+1-\ell)!} \Big(\frac{4}{\log\lambda} \Big)^\ell
        f_\gl( 2 w)^{j+1-\ell}.
    \end{align}
\end{lemma}

\begin{proof}
We start using the monotonicity of $F_\lambda$ and noting that for $\lambda> \text{$e^{2(\theta-\gamma)}$}$
and $u\geq 0$:
    \begin{align}\label{F}
    F_\gl(u+w)
     & \le F_\gl(w) = 
     \int^\infty_0\frac{ e^{-\gs w} \,\dd\gs}{\log(\gl+\gs/2)-(\theta-\gamma)} 
      \le 2 \int^\infty_0  \frac{ e^{-\gs w}}{\log\gl} \dd\gs
     = \frac{2}{w\log\lambda} .
    \end{align}
    We next split the integral on the left-hand side of \eqref{fF:eq1} into $\int_0^{2w}(\cdots) \dd u$ 
    and $\int_{2w}^{2}( \cdots )\dd u$, which we call $I$ and $II$, respectively. 
    We start by estimating integral $I$. By \eqref{F} we have, 
    \begin{equation}
    \begin{split}
        I & = 
        \int^{2w}_0 
        F_\gl(u+w)
        f_\gl(u)^j \dd u
        \le \frac{2}{w \log\gl}
        \int^{2w}_0
        f_\gl(u)^j \dd u\\
    \end{split}.
    \end{equation}
     By integration by parts we have,
    \begin{equation*}
    \begin{split}
        \int^{2w}_0 f_\gl(u)^jdu 
        & = 2w f_\gl(2w)^j
        -  j \int^{2w}_0 uf'_\gl(u)f_\gl(u)^{j-1} \dd u\\
        & \le 2w f_\gl (2w )^j 
           +j\frac{2}{\log \gl}\int^{2w}_0f_\gl(u)^{j-1} \dd u,
    \end{split}
    \end{equation*}
    where in the inequality we used (\ref{F}) and $-u f_\gl(u)'= uF_\gl(u)\le \frac{2}{\log\gl}$. Iterating this computation we have that, for $j\ge 1$,
\begin{align*}
    \int^{2w}_0 
    f_\gl(u)^j \dd u
    \le  2w \sum^j_{i=0}
    \frac{j!}{(j-i)!}
    \Big( \frac{2}{\log \gl} \Big)^i
    f_\gl(2w)^{j-i},
\end{align*}
and so 
\begin{align*}
    I \le \sum^j_{i=0}\frac{j!}{(j-i)!} \Big( \frac{4}{\log \gl} \Big)^{i+1}f_\gl(2w)^{j-i}.
\end{align*}
  On the other hand, $II$ is estimated as:
\begin{align*}
    II :=\int_{2w}^{2}  F_\gl(u+w)  f_\gl(u)^j \,\dd u
    \leq \int^{2}_{2w}
    F_\gl(u)f_\gl(u)^j \dd u
    = \frac{1}{j+1}f_\gl(2w)^{j+1},
\end{align*}
   where we used the monotonicity of $F$ and the fact that $f'=-F$. This completes the proof. 
\end{proof}

\begin{lemma}\label{u2}
    Fix $m\ge 2$. For all $1\le k\le m-1$ and $\sum^{m-k}_{i=1}u_i\le 2$ with $0\le u_i \le 2$:
    \begin{equation}
        \idotsint\limits_{\sum^m_{i=m-k+1}u_i \le 2} \,\,
        \prod^m_{r=m-k+1}   F_\gl\big(u_r+\frac{u_{r-1}}{2}\big) \,\dd u_r
        \le 
        \sum^k_{i=0}\frac{c^k_i}{(k-i)!}
        \ra
        \frac{4}{\log\gl}
        \rb^i
        f_\gl
        \ra
         u_{m-k}
        \rb^{k-i}
    \end{equation}
 where $c^k_i$ are combinatorial coefficients defined inductively by
\begin{align}\label{def of c^k_i}
     c^0_0 &=1; \, c^k_i=0\,\text{ for $i>k$} \qquad \text{and}\qquad
     c^{k+1}_i = \sum^{i}_{j=0}c^k_j\qquad \text{ for $i\le k+1$}.
\end{align}

\end{lemma}
\begin{proof}
    The proof here is an adaptation of the induction scheme of Lemma 3.9 in \cite{CZ23}. 
    When $k=1$, the statement follows from Lemma \ref{lemma:fF} for $j=0$ and $w=\frac{u_{r-1}}{2}$.  
    Assume the statement holds for some $k$ such that $1\le k\le m-2$. Then for $k+1$ we have by the
     inductive assumption that 
    \begin{equation}
    \begin{split}
    & \idotsint\limits_{\sum_{i=m-k}^m u_i\le 2}
     \prod^m_{r=m-k} F_\gl(u_r+\f{u_{r-1}}{2}) \, \prod^m_{r=m-k} \dd u_r\\
    & \le \int^{2}_0
        \Bigg( \,\,\,\,
        \idotsint\limits_{\sum_{i=m-k+1}^m u_i\le 2} \,\,
        \prod_{r=m-k+1}^m 
        F_\gl(u_r+\f{u_{r-1}}{2}) \dd u_r 
        \Bigg) \,
        F_\gl(u_{m-k}+\f{u_{m-k-1}}{2}) \, \dd u_{m-k}\\
        & \le 
        \int^{2}_0
        \sum^k_{i=0}\frac{c^k_i}{(k-i)!}
        \ra
        \frac{4}{\log\gl}
        \rb^i
        f_\gl
        \ra
        u_{m-k}
        \rb^{k-i}
        F_\gl(u_{m-k}+\f{u_{m-k-1}}{2}) \, \dd u_{m-k} \,.
    \end{split}
    \end{equation}
    Then by Lemma \ref{lemma:fF}, we bound the above by
    \begin{align*}
        & \sum^k_{i=0} \f{c^k_i}{(k-i)!} \ra \f{4}{\log\gl} \rb^i \,\ \sum^{k-i+1}_{l=0} \f{(k-i)!}{(k-i+1-l)!} \ra \f{4}{\log\gl} \rb^l f_\gl \ra u_{m-k-1} \rb^{k-i+1-l} \\
        & = \sum^k_{i=0} \sum^{k-i+1}_{l=0} \f{c^k_i}{(k+1-(i+l))!} \ra \f{4}{\log\gl} \rb^{i+l} f_\gl \ra u_{m-k-1} \rb^{k+1-(i+l)} 
    \end{align*}
    We introduce a new variable $n:=i+l$ to replace $\sum_{l=0}^{k-i+1}$ by $\sum_{n=i}^{k+1}$ and, thus, write the above as:
    \begin{align*}
        &  \sum^k_{i=0} \sum_{n=i}^{k+1} \f{c^k_i}{(k+1-n)!} \ra \f{4}{\log\gl} \rb^{n} f_\gl \ra u_{m-k-1} \rb^{k+1-n} 
         \le \sum_{n=0}^{k+1} \sum_{i=0}^{n} \f{c^k_i}{(k+1-n)!} \ra \f{4}{\log\gl} \rb^{n} f_\gl \ra u_{m-k-1} \rb^{k+1-n}
    \end{align*}
    By the definition of $c^k_i$ in \eqref{def of c^k_i} we complete the proof.
    \color{black}
\end{proof}
\begin{lemma}\label{after u, prop}
There exists constants $C>0$, depending on $h$,
 such that for all $\lambda> \ra e^{2(\theta-\gamma)} \vee 1 \rb$, we have
\begin{equation}\label{eq: I3}
\begin{split}
    \mathfrak{M}_\epsilon^{\theta,h}
    & \le C e^{2\gl} \log(\f{1}{\epsilon})
    \sum_{m\ge 0} {h \choose 2} \left[ {h\choose 2} -1 \right]^{m-1}
    \sum_{i=0}^{m}\frac{c_i^{m}}{(m-i)!}
    \ra
    \frac{4}{\log\gl}
    \rb^i
    f_\gl(\epsilon^2)^{m-i}.
\end{split}
\end{equation}
\end{lemma}
\begin{proof}
    For $m\ge 2$, recall \eqref{eq: I2}:
    \begin{align*}
        \mI^{(\lambda)}_{m,h,\epsilon} & \le {h \choose 2} \left[ {h\choose 2} -1 \right]^{m-1}
        \idotsint\limits_{\sum_i u_i \le 2\,,\, u_1>\epsilon^2}
        \frac{1}{u_1}
        \prod^m_{r=2}F_\gl \big(u_r+ \frac{u_{r-1}}{2} \big) \,\, \dd\vec u .
    \end{align*}
    By Lemma \ref{u2} for $k=m-1$ we have that:
    \begin{equation}\label{before u1}
    \begin{split}
    \mI^{(\lambda)}_{m,h,\epsilon}
    & \le 
    {h \choose 2} \left[ {h\choose 2} -1 \right]^{m-1}
    \int^2_{\epsilon^2}  
    \sum^{m-1}_{i=0}
    \frac{c_i^{m-1} f_\gl(u_1)^{m-1-i}}{(m-1-i)!}
    \Big(
    \frac{4}{\log\gl}
    \Big)^i \, \frac{\dd u_1}{u_1}\, .
    \end{split}
    \end{equation}
    Since $f_\lambda$ is decreasing we have that $f_\lambda(u_1)\leq f_\lambda(\epsilon^2)$, so for 
    $u_1\geq \epsilon^2$:
    \begin{align}\label{F(u1)}
        \int^2_{\epsilon^2} 
        f_\gl(u_1)^{m-1-i} \, \frac{\dd u_1}{u_1}
        & \le 
        \int^2_{\epsilon^2}
        f_\gl(\epsilon^2)^{m-1-i} \, \frac{\dd u_1}{u_1} 
        \le  C \log \Big( \f{1}{\epsilon} \Big) \, f_\gl(\epsilon^2)^{m-1-i} \,.
    \end{align}
    Therefore, we obtain:
    \begin{align*}
        &\mathfrak{M}_\epsilon^{\theta,h}  \le C e^{2\gl} \sum_{m\ge 0} \mI^{(\lambda)}_{m,h,\epsilon}\\
        & \le C e^{2\gl} \ra \mI^{(\lambda)}_{0,h,\epsilon} + \mI^{(\lambda)}_{1,h,\epsilon} + \log\ra \f{1}{\epsilon} \rb 
        \sum_{m\ge 2} 
        {h \choose 2} \left[ {h\choose 2} -1 \right]^{m-1} 
        \sum^{m-1}_{i=0}
        \frac{c_i^{m-1} f_\gl(\epsilon^2)^{m-1-i}}{(m-1-i)!}
        \Big(
        \frac{4}{\log\gl}
        \Big)^i \, 
        \rb\\
        & = C e^{2\gl} \ra \mI^{(\lambda)}_{0,h,\epsilon} + \mI^{(\lambda)}_{1,h,\epsilon} + \ra{h\choose 2}-1\rb \log\ra \f{1}{\epsilon} \rb 
        \sum_{m\ge 1} 
        {h \choose 2} \left[ {h\choose 2} -1 \right]^{m-1} 
        \sum^{m}_{i=0}
        \frac{c_i^{m} f_\gl(\epsilon^2)^{m-i}}{(m-i)!}
        \Big(
        \frac{4}{\log\gl}
        \Big)^i \, 
        \rb\, ,
    \end{align*}
    where in the last step we change the variable $m\mapsto m+1$. The result follows by recalling the definitions of 
     $\mI^{(\lambda)}_{0,h,\epsilon}$ and  $\mI^{(\lambda)}_{1,h,\epsilon}$ from Lemmas \ref{before replica} and 
     \ref{technical1}.
    \end{proof}


\subsection{Combinatorial coeffiients}
We will derive an exact formula for $c^m_i$ defined in \eqref{def of c^k_i}.
\begin{lemma}\label{combinatorial}
For $m\ge i\ge 0$, we have:
\begin{align}\label{eq: exact formula of c^m_i}
    c^m_i = \f{m-i+1}{i!}\f{(m+i)!}{(m+1)!} \le 4^m \,.
\end{align}
\end{lemma}
\begin{proof}
    Given the first equality, the inequality is obvious. Indeed, 
    \begin{align*}
        \f{m-i+1}{i!}\f{(m+i)!}{(m+1)!}
        \le 
        \f{(m+i)!}{i!m!}
        = {m+i \choose i}
        \le {2m \choose i}
        \le \sum_{0\le i\le 2m}{2m\choose i}
        \le 4^m.
    \end{align*}
    To prove the equality, first notice that by definition of $c^m_i$, we could simplify the recursion to:
    \begin{align}\label{new recursion}
        c^{m}_{i}
        =
        \sum^{i}_{j=0}c^{m-1}_j
        =
        c^{m-1}_{i} + \ind_{\{i>0\}}\sum^{i-1}_{j=0}c^{m-1}_j
        =
        c^{m-1}_{i} + c^{m}_{i-1}\ind_{\{i>0\}} \,,
    \end{align}
    for all $1\le i\le m$. It then suffices to verify that the equality of \eqref{eq: exact formula of c^m_i} solves the recursion formula \eqref{new recursion} with boundary conditions $c^0_0=1$ and $c^m_{m+1}=0$ for all $m\in\bbN_0$. Suppose first $m\ge 0$ with $0< i< m$; then:
    \begin{align*}
        c^m_i=\f{m-i+1}{i!} \f{(m+i)!}{(m+1)!} & = 
        \f{m-i+1}{i!} \f{(m+i)!}{(m+1)!} \ra \f{m-i}{m-i+1} \f{m+1}{m+i} + \f{m-i+2}{m-i+1} \f{i}{m+i} \rb \\
        & = \f{m-i}{i!} \f{(m+i-1)!}{m!} + \f{m-i+2}{(i-1)!} \f{(m+i-1)!}{(m+1)!}
        =c^{m-1}_i+c^m_{i-1},
    \end{align*}
    which is exactly \eqref{new recursion}. Next suppose that $0 < i=m$. In this case,
        \begin{align*}
        c^m_m = \f{m -m+1}{m!} \f{(m+m)!}{(m+1)!} = 0+ \f{m-m+2}{(m-1)!} \f{(m+m-1)!}{(m+1)!} 
        =c^{m-1}_m + c^m_{m-1} \,,
    \end{align*}
    which also agrees with \eqref{new recursion}. Lastly, for $i=0$ and $m\in\bbN_0$, we can readily check that $c^m_0=1$ from the claimed formula,
    which again coincides with \eqref{new recursion} with the initial condition $c^0_0=1$.
\end{proof}

\subsection{Final step}
We will next bound \eqref{eq: I3} by the upper bound claimed in Theorem \ref{central} and hence complete the proof. The following asymptotic behavior of $f_\gl$ will be useful:
\begin{lemma}\label{eq: asymptotic of f 1}
    Suppose that $\gl>e^{\theta-\gamma+\f{1}{2} }$. There exists $C>0$ and $\gd_\epsilon := \f{C}{\log\log\f{1}{\epsilon} } $
    such that for all $u\in (0, \epsilon^2]$:
    \begin{align}\label{eq:fe}
    f_\gl(u)\le (1+\gd_\epsilon)\log\log\frac{1}{u}.
    \end{align}
    \end{lemma}
    \begin{proof}
    Recall \eqref{eq: def of f}. Let $C_{\gt}:=e^{2\ra \log 2 + \theta-\gamma \rb }$. Without loss of generality,
    assume that $\epsilon$ is small enough, so that for $u\le \epsilon^2$, $1/u>C_\theta$. We next have,
    \begin{align}
        f_\gl(u)
        = & 
        \int^\infty_0  \frac{1}{\gs}\frac{1}{\log(\gl+\gs/2)-\theta-\gamma)} (e^{-\gs u}-e^{-2\gs}) \,\dd\gs \label{eq: f0}\\
        = & \label{eq: f1} 
        \int^{C_{\theta}}_0  \frac{1}{\gs}\frac{1}{\log(\gl+\gs/2)-(\theta-\gamma) } (e^{-\gs u}-e^{-2\gs})  \,\dd\gs \\
        & +\label{eq: f2}
        \int^{\frac{1}{u}}_{C_{\theta}}  \frac{1}{\gs}\frac{1}{\log(\gl+\gs/2)-(\theta-\gamma)} (e^{-\gs u}-e^{-2\gs}) \,\dd\gs\\
        & + \label{eq: f3}
        \int^\infty_\frac{1}{u} \frac{1}{\gs}\frac{1}{\log(\gl+\gs/2)-(\theta-\gamma)} (e^{-\gs u}-e^{-2\gs}) \, \dd\gs
    \end{align}
    We see that in \eqref{eq: f1}, the integrand is bounded. Indeed, given the assumption on $\lambda$, we have:
    \begin{align*}
        \frac{1}{\gs}\frac{1}{\log(\gl+\gs/2)-(\theta-\gamma)} (e^{-\gs u}-e^{-2\gs}) 
        & \le  \f{ 2 \ra e^{-\gs u}-e^{-2\gs} \rb }{\gs} 
        <C_1,
    \end{align*}
  for some finite constant $C_1$. 
    For \eqref{eq: f3}, we notice that:
    \begin{align*}
        \int^\infty_\frac{1}{u}  \frac{\dd\gs}{\gs}\frac{1}{\log(\gl+\gs/2)-(\theta-\gamma)} (e^{-\gs u}-e^{-2\gs})
        & \le 
        2 \int^\infty_\frac{1}{u} (e^{-\gs u}-e^{-2\gs}) \,   \frac{\dd\gs}{\gs}\\
        & \le
        2 \int^\infty_\frac{1}{u}  e^{-\gs u}\,  \frac{\dd\gs}{\gs}
        =
        2 \int^\infty_1   e^{-\gs } \frac{\dd\gs}{\gs} \, <C_2,
    \end{align*}
    for some finite constant $C_2$, given the assumption that $u\leq \epsilon^2$ and $\epsilon$ is assumed to 
    be small enough. 
We claim that the main growth in \eqref{eq: f0} comes from  \eqref{eq: f2}. To this end, we have:
    \begin{align}\label{eq: pre 1/slogs}
    \int^{\frac{1}{u}}_{C_{\theta}} 
    \frac{\dd\gs }{\gs}
    \frac{1}{\log(\gl+\gs/2)-(\theta-\gamma)} 
    (e^{-\gs u}-e^{-2\gs})
    & \le
    \int^{\frac{1}{u}}_{C_{\theta}} 
    \frac{\dd\gs }{\gs \ra \log(\gs/2)-(\theta-\gamma) \rb},
    \end{align}
    Notice that:
    \begin{align*}
        \int^{\frac{1}{u}}_{C_{\theta}} 
        \frac{\dd\gs}{\gs \ra \log(\gs/2)-(\theta-\gamma) \rb} - \int^{\frac{1}{u}}_{C_{\theta}}  \f{\dd\gs}{\gs \log\gs}
        & = \int^{\f{1}{u} }_{C_{\theta} } \f{1}{\gs} \ra \f{1}{\log\gs - \ra \log 2 + (\theta-\gamma) \rb } -\f{ 1 }{ \log\gs } \rb \dd\gs\\
        & \le \int^{\infty }_{C_{\theta} }\f{1}{\gs} \f{ \log 2+(\theta-\gamma) }{ \ra \log \gs \rb^2 -\ra \log 2+(\theta-\gamma) \rb \log\gs } \, \dd\gs\\
        & \le 2\ra \log 2+ (\theta-\gamma) \rb \int^{\infty }_{C_{\theta} }  \f{ \dd\gs }{ \gs \ra \log \gs \rb^2  }
        < C_3.
    \end{align*}
    for some finite constant $C_3$, where we also used the assumption that $C_\theta:=e^{2\ra \log 2 + \theta-\gamma \rb }$.
    Therefore, we can bound \eqref{eq: pre 1/slogs} by:
    \begin{align*}
      \int^{\frac{1}{u}}_{C_{\theta}} \f{1}{\gs\log\gs} \dd \gs +C_3   
    \leq \log\log \f{1}{u} + C,
       \end{align*}
       for some finite constant $C$.
    Putting the bounds for \eqref{eq: f1}-\eqref{eq: f3} together we obtain:
    \begin{align*}
        f_\gl(u) \le  \dlog\f{1}{u}+C \le (1+\f{C}{\dlog \f{1}{\epsilon^2}})\dlog \f{1}{u},
    \end{align*}
    for some $C>0$, when $u\leq \epsilon^2$, from which the result follows.
        \end{proof}

Now we are ready to prove Proposition \ref{thm: g upper bound}.
\begin{proof}[Proof of Proposition \ref{thm: g upper bound}]
We build upon \eqref{eq: I3}: 
\begin{align}
        \mathfrak{M}_\epsilon^{\theta,h}
        \le &  
        C e^{2\gl}
        \log \ra \f{1}{\epsilon} \rb
        \sum_{m\ge 0}
        \sum_{i: i\le m}
        {h\choose 2}
        \qa
        {h\choose 2}-1      
        \qb^{m-1} 
        c^m_i
        \ra \f{4}{\log\gl} \rb^i
        \frac{f^{m-i}_\gl(\epsilon^2)}{(m-i)!}\\
        = &  
        C e^{2\gl}
        \log \ra \f{1}{\epsilon} \rb
        \sum_{i\ge 0}
        \sum_{m: m \ge i}
        {h\choose 2}
        \qa
        {h\choose 2}-1      
        \qb^{m-1} 
        c^m_i
        \ra \f{4}{\log\gl} \rb^i
        \frac{f^{m-i}_\gl(\epsilon^2)}{(m-i)!}\\
        = &  \label{eq: m>k}
        C e^{2\gl}
        \log \ra \f{1}{\epsilon} \rb
        \sum_{i\ge 0}
        \sum_{m \geq i} \ind_{m \geq k}
        {h\choose 2}
        \qa
        {h\choose 2}-1      
        \qb^{m-1} 
        c^m_i
        \ra \f{4}{\log\gl} \rb^i
        \frac{f^{m-i}_\gl(\epsilon^2)}{(m-i)!}\\
        & + \label{eq: m<k}
        C e^{2\gl}
        \log \ra \f{1}{\epsilon} \rb
        \sum_{i\ge 0}
        \sum_{m\ge i}
        \ind_{m<k}
        {h\choose 2}
        \qa
        {h\choose 2}-1      
        \qb^{m-1} 
        c^m_i
        \ra \f{4}{\log\gl} \rb^i
        \frac{f^{m-i}_\gl(\epsilon^2)}{(m-i)!} \,,
\end{align}
where we break the sum over $m$ at $k=C_0 \dlog\f{1}{\epsilon}$ for some $C_0>0$ to be chosen later. We handle \eqref{eq: m>k} first. 
We use the bound $c^m_i \leq 4^m$ from \eqref{eq: exact formula of c^m_i}, to estimate \eqref{eq: m>k} by: 
\begin{align*}
        &  C e^{2\gl} \log(\f{1}{\epsilon})
        \sum_{i \ge 0}
        \sum_{m \ge i } \ind_{m\geq k}
        {h\choose 2} \qa {h\choose 2} -1 \qb^{m-1} 4^m
        \ra \f{4}{\log\gl} \rb^i
        \frac{f_\gl^{m-i}(\epsilon^2)}{(m-i)!} \\
        & \le C e^{2\gl} \log(\f{1}{\epsilon})
        \sum_{i \ge 0}
        \sum_{m \geq i } \ind_{m\geq k}
        \ra \f{h^2}{2} \rb^m 4^m
        \ra \f{4}{\log\gl} \rb^i
        \frac{f_\gl^{m-i}(\epsilon^2)}{(m-i)!} \\
        & = C e^{2\gl} \log(\f{1}{\epsilon})
        \sum_{i \ge 0}
        \sum_{m \ge i} \ind_{  m \ge k  }
        \ra \f{8h^2}{\log\gl} \rb^i
        \frac{\ra 2h^2f_\gl(\epsilon^2) \rb^{m-i}}{(m-i)!} \,.
\end{align*}
Now we introduce a new multiplier, $\mu>0$. As $\ind_{m\ge k}\le e^{(m-k)\mu}$, the above is bounded by:
\begin{equation}\label{small}
\begin{split}
    &  C e^{2\gl} \log(\f{1}{\epsilon})
    \sum_{i \ge 0}
    \sum_{m \ge i} 
    e^{(m-k)\mu}
    \ra \f{8 h^2}{\log\gl} \rb^i
    \frac{\ra 
    2 h^2 f_\gl(\epsilon^2)
    \rb^{m-i}
    }{(m-i)!}\\
    = \,&  C e^{2\gl} \log(\f{1}{\epsilon}) e^{-k\mu}
    \sum_{i\ge 0}
    \ra
    \f{8 e^\mu h^2}{\log\gl}
    \rb^i
    \sum_{m: m\ge i}
    e^{(m-i)\mu}
    \f{
    \ra
    2 h^2 f_\gl( \epsilon^2 )
    \rb^{m-i}
    }{(m-i)!} \,.
\end{split}
\end{equation}
We choose $\gl> \exp (8 e^\mu h^2)$ so that 
\begin{align}\label{eq: C_gl,h}
C \sum_{i\ge 0}
    \ra
    \f{8 e^\mu h^2}{\log\gl}
    \rb^i
    =: C_{\gl,h}<\infty.
\end{align}
Hence, if we also change variables $n:=m-i$, \eqref{small} is bounded by:
\begin{equation}\label{small2}
\begin{split}
    & C_{\gl,h} e^{2\gl} \log(\f{1}{\epsilon}) e^{-k\mu}
    \sum_{n\ge 0}
    e^{n\mu}
    \f{
    \ra
    2 h^2 f_\gl( \epsilon^2 )
    \rb^{n}
    }{n!}
    \le C_{\gl,h} e^{2\gl} \log(\f{1}{\epsilon}) e^{-k\mu} e^{ 2h^2e^\mu f_\gl(\epsilon^2) }
\end{split}
\end{equation}
We recall from \eqref{eq:fe} that $f_\gl(\epsilon^2)\leq(1+\gd_\epsilon)\dlog\f{1}{\epsilon}$ and 
that we chose $k=C_0\dlog\f{1}{\epsilon}$ where we will take $C_0>0$ such that:
\begin{align}\label{eq:C0}
C_0\mu - 2 e^\mu h^2 =: D >1.
\end{align}
Then (\ref{small2}) is bounded above by 
\begin{align}\label{eq: final upper bound of m>k}
C_{\gl,h} e^{2\gl} \log(\f{1}{\epsilon}) e^{-D \dlog(\f{1}{\epsilon})} 
= C_{\gl,h} e^{2\gl} \ra \log(\f{1}{\epsilon}) \rb^{1-D}
\rightarrow 0
\end{align}
as $\epsilon\rightarrow 0$.

Now we handle \eqref{eq: m<k}. By the equality part of \eqref{eq: exact formula of c^m_i}, \eqref{eq: m<k} becomes
\begin{equation}\label{eq: mid bound for m<k}
\begin{split}
    & C e^{2\gl}
    \log \ra \f{1}{\epsilon} \rb
    \sum_{i\ge 0}
    \sum_{m\ge i}
    \ind_{k>m}
        {h\choose 2}
    \qa
    {h\choose 2}-1      
    \qb^{m-1} 
    \f{m-i+1}{i!}\frac{(m+i)!}{(m+1)!}
    \ra \f{4}{\log\gl} \rb^i
    \frac{f^{m-i}_\gl(\epsilon^2)}{(m-i)!}\\
    & = C e^{2\gl}
    \log \ra \f{1}{\epsilon} \rb
    \sum_{0\le m < k}
    \sum_{0\le i\le m}
    {h\choose 2}
    \qa
    {h\choose 2}-1      
    \qb^{m-1} 
    \f{m-i+1}{i!}\frac{(m+i)!}{(m+1)!}
    \ra \f{4}{\log\gl} \rb^i
    \frac{f^{m-i}_\gl(\epsilon^2)}{(m-i)!}\\
    & = C e^{2\gl}
    \log \ra \f{1}{\epsilon} \rb
    \sum_{0\le m < k}
    \sum_{0\le i\le m}
    {h\choose 2}
    \qa
    {h\choose 2}-1      
    \qb^{m-1} 
    {m\choose i}
     \frac{m-i+1}{m!} \frac{(m+i)!}{(m+1)!}
    \ra \f{4}{\log\gl} \rb^i
    f^{m-i}_\gl(\epsilon^2)\\
    & =
    C e^{2\gl}
    \log \ra \f{1}{\epsilon} \rb \f{{h\choose 2}}{{h\choose 2}-1}
    \sum_{0\le m < k}
    \qa {h\choose 2}-1 \qb^m 
    \frac{1}{m!}
    \sum_{i: 0\le i\le m}
    {m\choose i}
    \frac{ (m-i+1) \, (m+i)!}{(m+1)!}
    \ra \f{4}{\log\gl} \rb^i
    f^{m-i}_\gl(\epsilon^2).
    \end{split}
    \end{equation}
    Notice that:
    \begin{align*}
    (m-i+1)\frac{(m+i)!}{(m+1)!}
    \le 
    \frac{(m+i)!}{m!}\le (2m)^i.
    \end{align*}
    Therefore, \eqref{eq: mid bound for m<k} is bounded by:
    \begin{align}
    & 
    C e^{2\gl}
    \log \ra \f{1}{\epsilon} \rb \f{{h\choose 2}}{{h\choose 2}-1}
    \sum_{0\le m\le k}
    \qa {h\choose 2}-1 \qb^m 
    \frac{1}{m!}
    \sum_{i: 0\le i\le m}
    {m\choose i}
    (2m)^i
    \ra \f{4}{\log\gl} \rb^i
    f^{m-i}_\gl(\epsilon^2)\\
    =&  \,C e^{2\gl} \label{eq: binomial}
    \log \ra \f{1}{\epsilon} \rb \f{{h\choose 2}}{{h\choose 2}-1}
    \sum_{0\le m\le k}
    \qa {h\choose 2}-1 \qb^m 
    \frac{1}{m!}
    \ra \f{8m}{\log\gl} + f_\gl(\epsilon^2) \rb^m\\
    \le &  \, C e^{2\gl}
    \log \ra \f{1}{\epsilon} \rb \f{{h\choose 2}}{{h\choose 2}-1}
    \sum_{0\le m < \infty}
    \qa {h\choose 2}-1 \qb^m 
    \frac{1}{m!}
    \ra \f{8k}{\log\gl} + f_\gl(\epsilon^2) \rb^m\\
    =&\, \label{eq: exponential}
    C e^{2\gl}
    \log \ra \f{1}{\epsilon} \rb \f{{h\choose 2}}{{h\choose 2}-1}
    e^{ \ra {h\choose 2}-1 \rb \ra f_\gl(\epsilon^2)+\f{8k}{\log\gl} \rb } \,.
\end{align}
where we have completed a binomial sum (resp. an exponential sum) to obtain \eqref{eq: binomial} (resp. \eqref{eq: exponential}). Again, we recall the asymptotic behavior of $f$ from  \eqref{eq:fe} and definition of $k$, and bound \eqref{eq: exponential} by:
\begin{equation}\label{eq: final bound for m<k}
\begin{split}
    & C e^{2\gl}
    \log \ra \f{1}{\epsilon} \rb \f{{h\choose 2}}{{h\choose 2}-1}
    e^{ \ra {h\choose 2}-1 \rb \ra (1+\gd_\epsilon)\dlog(\f{1}{\epsilon})+\f{8C_0\dlog(\f{1}{\epsilon})}{\log\gl} \rb }\\
    & =
    C e^{2\gl}
    \log \ra \f{1}{\epsilon} \rb \f{{h\choose 2}}{{h\choose 2}-1}
    e^{ \ra {h\choose 2}-1 \rb \ra 1+\gd_\epsilon + \f{8C_0}{\log\gl} \rb \dlog\f{1}{\epsilon} }\\
    & = C e^{2\gl}
    \f{{h\choose 2}}{{h\choose 2}-1} \ra \log \f{1}{\epsilon}  \rb^{{h\choose 2} \ra 
1+\gd_\epsilon +\f{8C_0}{\log\gl} \rb} \,.
\end{split}
\end{equation}

Putting \eqref{eq: final upper bound of m>k} and \eqref{eq: final bound for m<k} together, we obtain:
\begin{align}
    \mathfrak{M}_\epsilon^{\theta,h} & \le C e^{2\gl} \f{{h\choose 2}}{{h\choose 2}-1} \ra \log \f{1}{\epsilon} \rb^{ {h\choose 2} ( 1+\gd_\epsilon +\f{8C_0}{\log\gl} )} 
    + C_{\gl,h} e^{2\gl} \ra \log\f{1}{\epsilon} \rb^{1-D} \notag \\
    & \le C_h e^{2\gl} \ra \log \f{1}{\epsilon} \rb^{{h\choose 2} ( 1+\gd_\epsilon +\f{8C_0}{\log\gl} ) }
    = C_h \ra \log\f{1}{\epsilon} \rb^{ {h\choose 2} } \ra \log\f{1}{\epsilon} \rb^{ \f{2\gl}{\log\log\f{1}{\epsilon}} +  \gd_\epsilon+\f{8C_0  }{\log\gl}  }, \label{def: C}
\end{align}
To achieve the desired result we set $\gl=\gl_\epsilon$ and we require:
\begin{align*}
    \gd_{\epsilon,\gl} :=
    \f{2\gl_\epsilon}{\log\log\f{1}{\epsilon}} +  \gd_\epsilon+\f{8C_0  }{\log\gl_\epsilon}  = o(1),
\end{align*}
Recall from Lemma \ref{eq: asymptotic of f 1} that $\gd_\epsilon = O(1/\log\log\f{1}{\epsilon})$, so the optimising level of $\gl_\epsilon$ occurs when $\f{\gl_\epsilon}{\log\log\f{1}{\epsilon} } \sim \f{ 1 }{\log\gl_\epsilon}$. One choice would be $\gl_\epsilon=\f{\log\log\f{1}{\epsilon} }{\log\log\log\f{1}{\epsilon}}$. Substituting into \eqref{def: C}, we obtain for some $C_0'>0$:
\begin{align}\label{eq: optimal upper bound}
    \mathfrak{M}_\epsilon^{\theta,h} \le \ra \log\f{1}{\epsilon} \rb^{ {h\choose 2}  + \f{C_0'}{\log\log\log\f{1}{\epsilon} }  }.
\end{align}

\end{proof}

\section{Lower bound}\label{sec:lower}
We will again reduce the lower bound in Theorem \ref{central}  to a lower bound for the quantity $\bbE\Big[ \Big( \mathscr{Z}_{t}^\theta(g_{\epsilon^2}) \Big)^h\Big]$ for which it has been proven in \cite{CSZ23b} that
 there exists $\eta>0$, independent of $\epsilon$, such that
\begin{align}\label{GCI}
\bbE\Big[ \Big( 2 \mathscr{Z}_{t}^\theta(g_{\epsilon^2}) \Big)^h\Big]
\geq (1+\eta) \bbE\Big[ \Big( 2 \mathscr{Z}_{t}^\theta(g_{\epsilon^2}) \Big)^2\Big]^{h \choose 2}.
\end{align}
The lower bound in \eqref{thm:mom} then follows from the second moment asymptotic \eqref{2mom-as}.
We note that a weak version of inequality \eqref{GCI} is a consequence of the Gaussian Correlation Inequality \cite{R14, LM17}. More work
was required in \cite{CSZ23b} to obtain the uniform strict inequality.

In order to reduce the lower bound on $\bbE \qa \ra \mathscr{Z}^\gt_1(\cU_{B(0,R\epsilon)}) \rb \qb ^h$ to a lower bound as in \eqref{GCI}
we will bound $\cU_{B(0,\epsilon)}$ from below by $g_{\epsilon^2/2}\ind_{B(0,\epsilon)}$ and control the contribution from $\mathscr{Z}^\gt_1 \ra g_{\epsilon^2/2}\ind_{B^c(0,\epsilon)} \rb$.  This is summarised in the following lemma:
\begin{lemma}\label{thm: outer lower bound}
    For all $\rho\in (0,1)$, there exists $R>0$:
        \begin{align}\label{eq: outer lower bound}
        \bbE  \Big[ \Big( \mathscr{Z}^\gt_1 \big( g_{\epsilon^2/2} \ind_{ \{ |\cdot| \le R \epsilon \} } \big) \Big)^h \Big] 
        \ge (1-\rho) \bbE   \Big[ \Big(\mathscr{Z}^\gt_1 \big( g_{ \epsilon^2/2 } \big) \Big)^h\Big]  +   
        o\Big( \big(\log\f{1}{\epsilon} \big)^{h\choose 2} \Big)
    \end{align}

    where $|\cdot|$ represents the usual 2d Euclidean norm.
\end{lemma}
\begin{proof}
    Recall the fixed time marginals of the SHF from \eqref{one-time}. We then have that
    \begin{align}\label{eq: pre outer lower bound}
    \bbE  \qa \ra \mathscr{Z}^\gt_1 \ra g_{\epsilon^2/2} \rb \rb^h \qb
        -\bbE  \qa \ra \mathscr{Z}^\gt_1 \ra g_{\epsilon^2/2} \ind_{ \{ |\cdot| \le R \epsilon \} } \rb \rb^h \qb 
    \le \ra 2^h-1 \rb  \int_{|y_1|>R\epsilon} \prod_{i=1}^h g_{\epsilon^2/2}(y_i) \bbE \qa \prod_{i=1}^h \mathscr{Z}_1^\gt(\ind, \dd y_i)  \qb.
    \end{align}
    Observe that for $|x|>R\epsilon$:
    \begin{align*}
        g_{\epsilon^2/2}(x) 
        = \f{1}{\pi \epsilon^2} e^{-\f{|x|^2}{\epsilon^2} }
        = \f{1}{\pi \epsilon^2} \, e^{-\f{|x|^2}{2\epsilon^2} } \, e^{-\f{|x|^2}{2\epsilon^2} }
        \le \f{1}{\pi \epsilon^2} \, e^{-\f{R^2}{2} } \, e^{-\f{|x|^2}{2\epsilon^2} }
        = 2 e^{-\f{R^2}{2} } g_{\epsilon^2} (x).
    \end{align*}
    Substituting $x$ with $y_1$ and inserting this estimate in \eqref{eq: pre outer lower bound}, we obtain:
    \begin{align}\label{eq: outer upper 1}
        & \bbE  \qa \ra \mathscr{Z}^\gt_1 \ra g_{\epsilon^2/2} \rb \rb^h \qb
        -\bbE  \qa \ra \mathscr{Z}^\gt_1 \ra g_{\epsilon^2/2} \ind_{ \{ |\cdot| \le R \epsilon \} } \rb \rb^h \qb  \notag\\
        & \le 2 \ra 2^h-1 \rb e^{-\f{R^2}{2} } \int_{\bbR^{2h}} g_{\epsilon^2}(y_1) \prod_{i=2}^h g_{\epsilon^2/2}(y_i) \bbE \qa \prod_{i=1}^h \mathscr{Z}_1^\gt(\ind, \dd y_i)  \qb  \,.
    \end{align}
    We compare \eqref{eq: outer upper 1} and $\bbE  \qa \ra \mathscr{Z}^\gt_1 \ra g_{\epsilon^2/2} \rb \rb^h \qb$. We do so via chaos expansions.  Following the same procedure as in the derivation of \eqref{mom-formula-g} but with $\prod_{i=1}^h g_{\epsilon^2/2}(y_i)$ replaced by $g_{\epsilon^2}(y_1) \prod_{i=2}^h g_{\epsilon^2/2}(y_i)$  :
    \begin{align}\label{eq: 2epsilon}
    & \int_{\bbR^{2h}} g_{\epsilon^2}(y_1) \prod_{i=2}^h g_{\epsilon^2/2}(y_i) \bbE \qa \prod_{i=1}^h \mathscr{Z}_1^\gt(\ind, \dd y_i)  \qb  \notag\\
    & \le \sum_{m\ge 0}  \,\,\,\,\,\,\, (2\pi)^m \hskip -1.2cm\sumtwo{ \{i_1,j_1\},...,\{i_m,j_m\} \in \{1,...,h\}^2}
    {\text{with $\{i_k,j_k\} \neq \{i_{k+1},j_{k+1}\}$ for $k=1,...,m-1$}}    \notag \\
   &  \iint_{\substack{\epsilon^2
   \le a_1< b_1<...< a_m < b_m\le 1+2\epsilon^2 \\ x_1,y_1,...,x_m, y_m \in \R^2}}  \hskip -1.3cm
    \,\,\,
    g_{\f{a_1}{2}} (x_1)^2  \,\,
   \prod_{r=1}^m    G_\gt(b_r-a_r) g_{\frac{b_r-a_r}{4}}(y_r-x_r)  \,\, \ind_{\cS_{i_r,j_r}} \notag \\
    &\qquad \qquad  \times \Big( \prod_{1\le r\le m-1} g_{\frac{a_{r+1}-b_{\sfp(i_{r+1})}}{2}}(x_{r+1} - y_{\sfp(i_{r+1}) }) \, 
    g_{\frac{a_{r+1}-b_{\sfp(j_{r+1})}}{2}} (x_{r+1}- y_{\sfp(j_{r+1})}) \Big) \dd \vec\bx \,\dd \vec\by \,\dd \vec a \,\dd \vec b.
    \end{align}
    Note there is a difference at the integration range of $a_1,b_1,..., a_m, b_m$. They are integrated up to $1+2\epsilon^2$ due to $g_\epsilon^2(y_1)$, which pushes the upper bound up by $2\epsilon^2$ instead of $\epsilon^2$. 

    We also recall from \eqref{mom-formula-g} :
    \begin{align}\label{eq: epsilon}
    & \bbE  \qa \ra \mathscr{Z}^\gt_1 \ra g_{ \epsilon^2/2 }  \rb \rb^h \qb\notag\\
    & = \sum_{m\ge 0}  \,\,\,\,\,\,\, (2\pi)^m \hskip -1.2cm\sumtwo{ \{i_1,j_1\},...,\{i_m,j_m\} \in \{1,...,h\}^2}
    {\text{with $\{i_k,j_k\} \neq \{i_{k+1},j_{k+1}\}$ for $k=1,...,m-1$}}    \notag \\
   &  \iint_{\substack{\epsilon^2
   \le a_1< b_1<...< a_m < b_m\le 1+\epsilon^2 \\ x_1,y_1,...,x_m, y_m \in \R^2}}  \hskip -1.3cm
     \,\,\,
    g_{\f{a_1}{2}} (x_1)^2  \,\,
   \prod_{r=1}^m    G_\gt(b_r-a_r) g_{\frac{b_r-a_r}{4}}(y_r-x_r)  \,\, \ind_{\cS_{i_r,j_r}} \notag \\
    &\qquad \qquad  \times \Big( \prod_{1\le r\le m-1} g_{\frac{a_{r+1}-b_{\sfp(i_{r+1})}}{2}}(x_{r+1} - y_{\sfp(i_{r+1}) }) \, 
    g_{\frac{a_{r+1}-b_{\sfp(j_{r+1})}}{2}} (x_{r+1}- y_{\sfp(j_{r+1})}) \Big) \dd \vec\bx \,\dd \vec\by \,\dd \vec a \,\dd \vec b.
    \end{align}
    Note that the only difference with \eqref{eq: 2epsilon}
     is at the integration range of $a_1,b_1,...,a_m,b_m$. Therefore, we have that:
    \begin{align}\label{eq: g+D}
        \bigg| \int_{\bbR^{2h}} g_{\epsilon^2}(y_1) \prod_{i=2}^h g_{\epsilon^2/2}(y_i) \bbE \qa \prod_{i=1}^h \mathscr{Z}_1^\gt(\ind, \dd y_i)  \qb  
        - \bbE  \qa \ra \mathscr{Z}^\gt_1 \ra g_{ \epsilon^2/2 }  \rb \rb^h \qb \bigg|
        \le \mathscr{D},
    \end{align}
    where 
    \begin{align*}
    \mathscr{D} & := 
    \sum_{m\ge 0}  \,\,\,\,\,\,\, (2\pi)^m \hskip -1.2cm\sumtwo{ \{i_1,j_1\},...,\{i_m,j_m\} \in \{1,...,h\}^2}
    {\text{with $\{i_k,j_k\} \neq \{i_{k+1},j_{k+1}\}$ for $k=1,...,m-1$}}    \notag \\
   &  \iint_{\substack{\epsilon^2
   < a_1< b_1<...< a_m<b_m \le 1+2\epsilon^2 \\ 1+\epsilon^2<a_i\le 1+2\epsilon^2 \text{ or } 1+\epsilon^2<b_i\le 1+2\epsilon^2 \text{ for some $i=1,...,m$}\\ x_1,y_1,...,x_m, y_m \in \R^2}}  \hskip -1.3cm
    \,\,\,
    g_{\f{a_1}{2}} (x_1)^2  \,\,
   \prod_{r=1}^m    G_\gt(b_r-a_r) g_{\frac{b_r-a_r}{4}}(y_r-x_r)  \,\, \ind_{\cS_{i_r,j_r}} \notag \\
    &\qquad \qquad  \times \Big( \prod_{1\le r\le m-1} g_{\frac{a_{r+1}-b_{\sfp(i_{r+1})}}{2}}(x_{r+1} - y_{\sfp(i_{r+1}) }) \, 
    g_{\frac{a_{r+1}-b_{\sfp(j_{r+1})}}{2}} (x_{r+1}- y_{\sfp(j_{r+1})}) \Big) \dd \vec\bx \,\dd \vec\by \,\dd \vec a \,\dd \vec b .
    \end{align*}
    We will control $\mathscr{D}$. First, notice that the constraint 
    \begin{align*}
        1+\epsilon^2<a_i\le 1+2\epsilon^2 \text{ or } 1+\epsilon^2<b_i\le 1+2\epsilon^2 \text{ for some $i=1,...,m$},
    \end{align*}
    implies $1+\epsilon^2< b_m \le 1+2\epsilon^2$. Therefore, we relax the constraint in $\mathscr{D}$ to obtain:
    \begin{align}\label{eq: outer upper bound bm}
    \mathscr{D} & \le 
    \sum_{m\ge 0}  \,\,\,\,\,\,\, (2\pi)^m \hskip -1.2cm\sumtwo{ \{i_1,j_1\},...,\{i_m,j_m\} \in \{1,...,h\}^2}
    {\text{with $\{i_k,j_k\} \neq \{i_{k+1},j_{k+1}\}$ for $k=1,...,m-1$}}    \notag \\
   &  \iint_{\substack{\epsilon^2
   < a_1< b_1<...< a_m \le 1+2\epsilon^2 \\ ((1+\epsilon^2) \vee a_m) < b_m \le 1+2\epsilon^2 \\ x_1,y_1,...,x_m, y_m \in \R^2}}  \hskip -1.3cm
     \,\,\,
    g_{\f{a_1}{2}} (x_1)^2  \,\,
   \prod_{r=1}^m    G_\gt(b_r-a_r) g_{\frac{b_r-a_r}{4}}(y_r-x_r)  \,\, \ind_{\cS_{i_r,j_r}} \notag \\
    &\qquad \qquad  \times \Big( \prod_{1\le r\le m-1} g_{\frac{a_{r+1}-b_{\sfp(i_{r+1})}}{2}}(x_{r+1} - y_{\sfp(i_{r+1}) }) \, 
    g_{\frac{a_{r+1}-b_{\sfp(j_{r+1})}}{2}} (x_{r+1}- y_{\sfp(j_{r+1})}) \Big) \dd \vec\bx \,\dd \vec\by \,\dd \vec a \,\dd \vec b.
    \end{align}
    We first perform the integration over $y_m$:
    \begin{align*}
        \int_{\bbR^2} g_{ \f{b_m-a_m}{4} }(y_m-x_m) \dd y_m = 1.
    \end{align*}
    By Proposition \ref{asymptotic of G}, it is not difficult to see there exists $C=C_\gt>0$ such that for all $\epsilon\in(0,1)$:
    \begin{align}\label{eq: int of bm}
        \int^{1+2\epsilon^2}_{a_m \vee (1+\epsilon^2)} G_\gt (b_m-a_m) \dd b_m \le  \f{C}{\log\f{1}{\epsilon}} .
    \end{align}
    So \eqref{eq: outer upper bound bm} becomes:
    \begin{align}\label{eq: after bm}
    \mathscr{D} & \le \f{C}{\log\f{1}{\epsilon}}
    \sum_{m\ge 0}  \,\,\,\,\,\,\, (2\pi)^m \hskip -1.2cm\sumtwo{  \{i_1,j_1\},...,\{i_m,j_m\} \in \{1,...,h\}^2}
    {\text{with $\{i_k,j_k\} \neq \{i_{k+1},j_{k+1}\}$ for $k=1,...,m-1$}}    \notag \\
   &  \iint_{\substack{\epsilon^2
   < a_1< b_1<...< a_m \le 1+2\epsilon^2  \\ x_1,y_1,...,x_m \in \R^2}}  \hskip -1.3cm
     \,\,\,
    g_{\f{a_1}{2}} (x_1)^2  \,\,
   \prod_{r=1}^{m-1}    G_\gt(b_r-a_r) g_{\frac{b_r-a_r}{4}}(y_r-x_r)  \,\, \ind_{\cS_{i_r,j_r}} \notag \\
    &\qquad \qquad  \times \Big( \prod_{1\le r\le m-1} g_{\frac{a_{r+1}-b_{\sfp(i_{r+1})}}{2}}(x_{r+1} - y_{\sfp(i_{r+1}) }) \, 
    g_{\frac{a_{r+1}-b_{\sfp(j_{r+1})}}{2}} (x_{r+1}- y_{\sfp(j_{r+1})}) \Big) \dd \vec\bx \,\dd \vec\by \,\dd \vec a \,\dd \vec b.
    \end{align}
    Then, following the same computation from Lemma \ref{lemma: sum over spatial variables} onwards in the upper bound section, we obtain:
    \begin{align*}
        \mathscr{D} \le \f{1}{\log\f{1}{\epsilon}}\big( \log\f{1}{\epsilon}  \big)^{ {h\choose 2}(1+o(1)) }
        = o \Big( \big( \log\f{1}{\epsilon} \big)^{{h\choose 2}} \Big).
    \end{align*}
    Combining this with \eqref{eq: outer upper 1} and \eqref{eq: g+D} by choosing $R>\sqrt{2\log \f{2 \ra 2^h-1 \rb }{\rho}}$, we obtain the bound \eqref{eq: outer lower bound}.
\end{proof}
\begin{proof}[Proof of the Lower bound in \eqref{central}]
We will first prove that for some fixed $R>0$, there exists $C=C(\gt,h)>0$ such that
\begin{align*}
    \bbE \qa \ra \mathscr{Z}^\gt_1(\cU_{B(0,R\epsilon)}) \rb^h \qb
    \ge C \ra \log\f{1}{\epsilon} \rb^{h\choose 2},
\end{align*}
and from this we will deduce the statement for $\cU_{B(0,\epsilon)}$. 
For $R>1$ we have:
\begin{align*}
    \cU_{B(0,R\epsilon)}(\cdot) = \f{1}{\pi R^2 \epsilon^2} \ind_{B(0,R \epsilon)}(\cdot)
    \ge \f{1}{R^2} g_{\epsilon^2/2} \ind_{B(0,R\epsilon)} (\cdot).
\end{align*}
Therefore,
\begin{align}
    &\bbE \qa \ra \mathscr{Z}^\gt_1(\cU_{B(0,R\epsilon)}) \rb^h \qb
    \ge \f{1}{R^{2h}}\bbE \qa \ra \mathscr{Z}^\gt_1( g_{\epsilon^2/2} \ind_{B(0,R\epsilon)} ) \rb^h \qb \,. \notag 
\end{align}
By Lemma \ref{thm: outer lower bound}, for any $\rho\in(0,1)$, there exists $R>0$:
\begin{align*}
    \bbE \qa \ra \mathscr{Z}^\gt_1(\cU_{B(0,R\epsilon)}) \rb^h \qb
    & \ge \f{1}{R^{2h}}\ra (1-\rho) \bbE \qa \ra \mathscr{Z}^\gt_1( g_{\epsilon^2/2} ) \rb^h \qb
    + o \ra \ra \log \f{1}{\epsilon} \rb^{h\choose 2} \rb \rb \,.
\end{align*}
By \eqref{2mom-as} and \eqref{GCI}, there exists $C=C(\gt,h)$ such that $\bbE \qa \ra \mathscr{Z}^\gt_1( g_{\epsilon^2/2} ) \rb^h \qb \ge C \ra\log\f{1}{\epsilon}\rb^{h\choose 2}$. So we obtain the bound:
\begin{align*}
    \bbE \qa \ra \mathscr{Z}^\gt_1(\cU_{B(0,R\epsilon)}) \rb^h \qb
    \ge \f{C_{\rho, \gt, h}}{ R^{2h} }  \ra \log\f{1}{\epsilon} \rb^{h\choose 2}.
\end{align*}
Now, 
\begin{align*}
    \bbE \qa \ra \mathscr{Z}^\gt_1(\cU_{B(0,\epsilon)}) \rb^h \qb
    = \bbE \qa \ra \mathscr{Z}^\gt_1(\cU_{B(0,R \times \f{\epsilon}{R} )}) \rb^h \qb
    \ge C_{\rho, \gt, h, R} \ra \log\f{R}{\epsilon} \rb^{h\choose 2}  ,
\end{align*}
and we complete the proof.
\end{proof}
\vskip 4mm
{\bf Acknowledgments.} We thank Francesco Caravenna and Rongfeng Sun for useful comments.
\vskip 4mm
{\bf Funding.} The authors have no relevant financial or non-financial interests to disclose.
\vskip 4mm
{\bf Data Availibility Statement.} The manuscript has no associated data.

\end{document}